%% file: root.tex
\documentclass[letterpaper, 10 pt, conference]{ieeeconf}  

\input{Packages_NewCommands}

\IEEEoverridecommandlockouts
\begin{document}

\title{\LARGE \bf
Over-Approximating Minimizer Sets of Constrained Convex Programs with Parametric Uncertainty via Reachability Analysis
}

\author{
Brendan Gould$^{1\star}$,
Chih-Yuan Chiu$^{1\star}$, Antoine P. Leeman$^{2}$, \\ Kyriakos G. Vamvoudakis$^1$, Samuel Coogan$^1$, Glen Chou$^1$
\thanks{$^\star$Equal contribution. $^{1}$Georgia Institute of Technology, Atlanta, GA, USA. \texttt{\{bgould6\} at gatech dot edu}. $^{2}$ETH Zurich, IDSC, Switzerland. }
}

\maketitle

\thispagestyle{empty}
\pagestyle{empty}




\input{0_Abstract}


\input{1_Introduction}

\input{2_Related_Work}
\input{3_Problem_Statement}

\input{4_Preliminaries}

\input{5_Methods}

\input{6_Experiments}

\input{7_Conclusion_Future_Work}


\bibliographystyle{IEEEtran}
\bibliography{references_trunc}


\section*{Acknowledgement}
The authors thank Shuyu Zhan for his assistance with the simulation experiments.






\end{document}

%% file: Packages_NewCommands.tex
\usepackage{booktabs}
\usepackage{adjustbox}

\usepackage[top=20.2mm,bottom=18mm,left=16.9mm,right=20.2mm]{geometry}

\usepackage{wrapfig}
\usepackage{lipsum}
\usepackage{amssymb}
\usepackage{amsmath}
\usepackage{siunitx}
\usepackage{mathtools}
\usepackage{hyperref}
\let\labelindent\relax
\usepackage[inline]{enumitem}
\usepackage[ruled,vlined,linesnumbered]{algorithm2e}
\usepackage[noend]{algpseudocode}
\usepackage{ifthen}
\usepackage{dsfont}
\usepackage{balance}
\usepackage{dirtytalk}
\usepackage{afterpage}
\usepackage{changepage}
\usepackage{bm}
\usepackage{cite}

\usepackage{times}

\usepackage{svg}

\allowdisplaybreaks

\newcommand{\A}{\mathcal{A}}

\newcommand{\B}{\mathcal{B}}
\newcommand{\C}{\mathcal{C}}

\newcommand{\diag}{\text{diag}}
\newcommand{\blockDiag}{\text{blkdiag}}
\newcommand{\E}{\mathbb{E}}

\newcommand{\K}{\mathcal{K}}

\newcommand{\N}{\mathbb{N}}

\newcommand{\ra}{\rightarrow}
\newcommand{\R}{\mathbb{R}}

\newcommand{\Z}{\mathcal{Z}}

\newcommand{\obj}{J}
\newcommand{\objSq}{J_{\mathrm{sq}}}
\newcommand{\objSqCons}{J_{\mathrm{sqc}}}
\newcommand{\objLQR}{J_{\mathrm{LQR}}}
\newcommand{\param}{\theta}
\newcommand{\paramDim}{d}
\newcommand{\paramSet}{\Theta}
\newcommand{\paramSetSq}{\Theta_{\mathrm{sq}}}
\newcommand{\paramSetLQR}{\Theta_{\mathrm{LQR}}}

\newcommand{\var}{\xi}
\newcommand{\varMap}{\phi}
\newcommand{\varOpt}{\var^\star}
\newcommand{\varDim}{n}

\newcommand{\varOptSet}{S_\var^\star}

\newcommand{\varDual}{\lambda}
\newcommand{\varDualOpt}{\varDual^\star}

\newcommand{\outputs}{y}

\newcommand{\varStacked}{\boldsymbol{\var}}

\newcommand{\ratePGDStacked}{\boldsymbol{\ratePGD}}
\newcommand{\outputStacked}{\boldsymbol{y}}

\newcommand{\varNominal}{\hat\var}
\newcommand{\paramNominal}{\hat\param}
\newcommand{\ratePGDNominal}{\hat\ratePGD}

\newcommand{\ratePGDNominalStacked}{\hat{\boldsymbol{\ratePGD}}}

\newcommand{\paramError}{\Delta\param}
\newcommand{\ratePGDError}{\Delta\ratePGD}
\newcommand{\ratePGDErrorStacked}{\Delta\ratePGDStacked}

\newcommand{\convergenceRadius}{\beta}

\newcommand{\ballUnit}{\mathbb{B}}

\newcommand{\linearizationError}{r}
\newcommand{\disturbance}{d}
\newcommand{\disturbanceStacked}{\boldsymbol{\disturbance}}

\newcommand{\linearizationErrorSmoothed}{\tilde{\linearizationError}}
\newcommand{\disturbanceSmoothed}{\tilde{\disturbance}}

\newcommand{\regularizerSystemResponse}{H}

\newcommand{\smoothnessConstant}{L}
\newcommand{\strongConvexityConstant}{m}

\newcommand{\ratePGD}{\alpha}
\newcommand{\ratePGDMin}{c_\ratePGD}
\newcommand{\ratePGDMax}{C_\ratePGD}
\newcommand{\convergenceRatePGD}{\gamma}
\newcommand{\dynamicsPGD}{f_{\text{PGD}}}
\newcommand{\dynamicsPGDSq}{f_{\mathrm{PGD, sq}}}
\newcommand{\dynamicsPGDLQR}{f_{\mathrm{PGD, LQR}}}

\newcommand{\feedbackMap}{\pi}
\newcommand{\feedbackGain}{K}
\newcommand{\setFeedbackGains}{S_\feedbackGain}
\newcommand{\feedbackGainStacked}{\K}

\newcommand{\downshift}{\Z}

\newcommand{\systemResponse}{\Phi}

\newcommand{\jacobianState}{A}
\newcommand{\jacobianControl}{B}
\newcommand{\jacobianStateStacked}{\A}
\newcommand{\jacobianControlStacked}{\B}
\newcommand{\outputMatrixStacked}{\C}

\newcommand{\jacobianStateSmoothed}{\tilde{\jacobianState}}
\newcommand{\jacobianControlSmoothed}{\tilde{\jacobianControl}}

\newcommand{\proj}{\text{proj}}

\newcommand{\constraintSet}{{S_{\mathcal{C}}}}

\newcommand{\iterationHorizon}{N}
\newcommand{\LQRHorizon}{T}
\newcommand{\LQRIterate}{t}

\newcommand{\state}{\zeta}
\newcommand{\stateNominal}{\hat\state}
\newcommand{\stateStacked}{\boldsymbol{\state}}
\newcommand{\stateNominalStacked}{\hat{\boldsymbol{\state}}}
\newcommand{\stateError}{\Delta\state}
\newcommand{\stateErrorStacked}{\Delta\stateStacked}

\newcommand{\dynamicsFullState}{f}

\newcommand{\outputNoise}{v}

\newcommand{\tubeSize}{\tau}
\newcommand{\tubeSizeStacked}{\boldsymbol{\tau}}

\newcommand{\varInitialSet}{S_\var^0}

\newcommand{\hessianBound}{\mu}
\newcommand{\hessianBoundStacked}{\boldsymbol{\mu}}
\newcommand{\hessianBoundSmoothed}{\tilde{\hessianBound}}
\newcommand{\hessianBoundSmoothedStacked}{\tilde{\hessianBoundStacked}}

\newcommand{\varForwardReachableSet}[1]{S_{\var, {#1}}}

\newcommand{\varConvexHull}{S_{\var, \text{CH}}}
\newcommand{\conv}{\text{conv}}

\newcommand{\diam}{\text{diam}}
\newcommand{\dist}{\text{dist}}

\newcommand{\outputMatrix}{C}
\newcommand{\zero}{\textbf{0}}

\newcommand{\ratePGDBoundSlope}{a}
\newcommand{\ratePGDBoundOffset}{b}

\newcommand{\samplingRadius}{\delta}

\newcommand{\samplingRadiusMax}{C_\samplingRadius}
\newcommand{\samplingVariable}{q}

\newcommand{\dynamicsPGDSmoothed}{\tilde{f}_{\text{PGD}}}
\newcommand{\dynamicsFullStateSmoothed}{{\tilde{\dynamicsFullState}}}

\newcommand{\exampleFunction}{h}
\newcommand{\exampleFunctionSmoothed}{\tilde{\exampleFunction}}
\newcommand{\varExample}{z}

\newcommand{\exampleInputDim}{m}
\newcommand{\exampleOutputDim}{{\overline{m}}}

\newcommand{\lipschitz}{\ell}
\newcommand{\exampleConstraintSet}{S_\exampleFunction}

\newcommand{\unif}{\text{Unif}}

\newcommand{\LQRStateCost}{Q}
\newcommand{\LQRControlCost}{R}
\newcommand{\LQRStateDyn}{A}
\newcommand{\LQRControlDyn}{B}

\newcommand{\SqcQuadScale}{c_1}
\newcommand{\SqcLinScale}{c_2}
\newcommand{\SqcConsVal}{c_3}
\newcommand{\SqcLipschitz}{\ell_{\mathrm{sqc}}}

\newtheorem{remark}{Remark}
\newtheorem{assumption}{Assumption}
\newtheorem{proposition}{Proposition}

\newtheorem{theorem}{Theorem}


%% file: 0_Abstract.tex
\begin{abstract}
\looseness-1
We
study the set of solutions to a parameterized, strongly convex optimization problem whose cost depends on uncertain, bounded parameters.
We compute
a certified outer approximation of the corresponding set of optimizers, 
using convergence properties of the projected gradient descent (PGD) algorithm for convex programs. Concretely, by treating the cost parameter as constant but unknown, we interpret the PGD iterates as an uncertain dynamical system and analyze its forward reachable sets. 
Since PGD converges exponentially to the unique optimizer for each fixed parameter, these reachable sets provide 
outer approximations of the optimizer set, with an explicit error bound that decays exponentially with the iteration count. 
We apply system-level synthesis (SLS) on the PGD dynamics to optimize the step-size sequence and obtain reachable-set over-approximations. 
Our method outperforms existing baselines in over-approximating, with low conservativeness, the minimizer sets of convex programs with uncertain costs
and high-dimensional decision variables.
\end{abstract}

%% file: 1_Introduction.tex
\section{Introduction}
\label{sec: Introduction}

Across a broad range of controls and robotics contexts, optimization is a critical component of planning and decision-making.
Frequently, the data informing these decisions is not known exactly, but subject to many types of \emph{uncertainty}.
For instance, to operate safely in dense traffic, an autonomous vehicle must anticipate the motion of nearby cars and pedestrians, whose intent
can only be partially inferred from observations~\cite{englert2017inverse, Peters2023OnlineandOfflineLearningofPlayerObjectivesfromPartialObservationsinDynamicGames}.
When modeling the motion of such agents as cost minimization~\cite{englert2017inverse, Menner2021ConstrainedInverseOptimalControlWithApplicationtoaHumanManipulationTask}, intent uncertainty induces uncertainty in their planned motion, which could cause unsafe behavior.
Separately, quasistatic \cite{li2026certified} and deformable object \cite{li2025convex} contact dynamics models in robotic manipulation can be expressed as the solution to convex conic programs, where uncertainty in the mass matrix manifests as uncertainty in the cost function.
Towards studying robust decision forecasting,
we present a method for over-approximating the set of minimizers of a convex program whose cost is characterized by uncertain parameters.
By robustly bounding the solution set for 
all possible parameter realizations,
our work 
provides valuable guidance for safe downstream planning.

\looseness-1The tractability of robustly over-approximating the minimizer set varies significantly depending on the structural properties of the objective function, constraint set, and parameter set. 
Closed-form expressions for the minimizer set are generally obtainable only under restrictive assumptions. 
Sensitivity analysis techniques
can be applied to bound the minimizer's variation
as the parameter changes,
but such bounds are often highly conservative (see \ref{subsubsec:baseline}), limiting their utility for downstream tasks.
Related work in set-based reachability and robust optimal control studies how uncertainty propagates through dynamical systems \cite{althoff2008reachability, schafer2025robust}. However, these works seek a single solution minimizing worst-case cost under uncertainty, whereas we aim to over-approximate the full set of minimizers across all parameter realizations.
\looseness-1To overcome the limitations of existing methods, we use system-level synthesis (SLS) to compute finite-time forward reachable-set over-approximations of projected gradient descent (PGD) dynamics with low conservativeness. We then leverage known convergence properties of the PGD dynamics 
\cite{WrightRecht2022OptimizationForDataAnalysis} 
to convert these reachable sets into certified outer approximations of the corresponding minimizer set.
Concretely, the contribution of this paper is fourfold, as follows:
\begin{enumerate}
    \item We prove that, under standard assumptions on the objective, constraint, and parameter sets, the PGD dynamics' (true) forward reachable sets converge exponentially to the minimizer set. 
    We thus recast over-approximating the minimizer set as computing forward reachable sets of the PGD dynamics with a fixed but unknown parameter.
    \item For \textit{differentiable} PGD dynamics, we over-approximate forward reachable sets of PGD trajectories with the SLS-based reachability framework~\cite{Leeman2025RobustNonlinearOptimalControlviaSLS} for smooth nonlinear dynamics, scaling tractably to high-dimensional systems with low over-approximation error 
    \cite{Leeman2025RobustNonlinearOptimalControlviaSLS, leeman2026vision, furieri2019input, fang2026safe}.
    \item To compute reachable sets for \textit{non-differentiable} PGD dynamics, we apply the sampling-based smoothing method in \cite{Flaxman2005OnlineConvexOptimization} to construct a differentiable approximation of PGD for which the SLS techniques in \cite{Leeman2025RobustNonlinearOptimalControlviaSLS} apply.
    %
    \item 
    In numerical simulations,
    our method outperforms existing sensitivity analysis baselines in constructing over-approximations, with low conservativeness, of minimizer sets of optimization programs with parametric cost uncertainty. 
    Our method scales tractably to optimization programs with up to 64-dimensional variables.
\end{enumerate}

%% file: 2_Related_Work.tex
\looseness-1Our work is conceptually related to the field of robust optimization, which studies decision-making under worst-case realizations of uncertain data \cite{bertsimas2004price, ben1998robust, bertsimas2011theory}.
However, our aim differs in that
instead of identifying a single robust decision, we seek to characterize the set of decisions that are optimal for some admissible parameter value. 
In this sense, prior work on 
multi-parametric programming
and solution-map sensitivity analysis
are more closely aligned with our objective, since such works study how minimizers change as problem data vary \cite{bonnans2013perturbation, dontchev2009implicit, shapiro2005sensitivity, jones2007lexicographic}. 
Our contribution departs from this literature by focusing on a global over-approximation of the full minimizer set induced by parameter uncertainty, explicitly leveraging reachability techniques from control theory~\cite{bravo2005computation, immrax}.

\textbf{Notation}:
For $m, n \in \mathbb{Z}^+$, we define $[n] := \{1, \cdots, n\}$ and $[0, n] := \{0\} \cup [n]$, and denote the $n \times n$ identity matrix, $n \times n$ zero matrix, $m \times n$ zero matrix, and the zero vector in $\R^n$ by $I_{n \times n}$, $O_n$, $O_{m \times n}$, and $\zero_n$ respectively, with indices omitted when clear from context.
We denote the block-diagonal matrix constructed from given matrix sub-blocks $M_1, \cdots, M_n$ by $\blockDiag\{M_1, \cdots, M_n\}$; when $M_1, \cdots, M_n$ are all scalar, we write $\diag\{\cdot\}$ instead of $\blockDiag\{\cdot\}$.
Given $p \in [1, \infty]$ and $M \in \R^{n \times d}$, and we denote the unit $p$-norm ball in $\R^d$ by $\ballUnit_p^d$, and we define $M \ballUnit_p^d := \{Mx: x \in \ballUnit_p^d \}$.
Given a set $S \subseteq \R^n$, 
we define the \textit{diameter} of $S$ by $\diam(S) := \sup_{s, s' \in S} \Vert s - s'\Vert_2$ and the \textit{convex hull} of $S$ by $\conv(S)$.
Given sets $S, S' \subseteq \R^n$, we define the \textit{Minkowski sum} of $S$ and $S'$ by $S \oplus S' := \{s + s': s \in S, s' \in S' \}$, and the Hausdorff distance between $S$ and $S'$ by $\dist(S, S') := \inf \{ \epsilon > 0: S \subseteq S' \oplus \epsilon I_n \ballUnit_2^n, S' \subseteq S \oplus \epsilon I_n \ballUnit_2^n \}$.

%% file: 3_Problem_Statement.tex
\section{Problem Statement}
\label{sec: Problem Statement}




Consider an objective $\obj: \R^\varDim \times \R^\paramDim \ra \R$ mapping a variable $\var \in \constraintSet \subseteq \R^\varDim$ and a parameter $\param \in \paramSet \subset \R^\paramDim$ to a real number.
We impose the following regularity assumptions.
\begin{assumption}
\label{Assump: Objective Properties}
All partial second derivatives of $\obj(\var, \param)$ exist and are Lipschitz continuous. 
There exist scalars
$\strongConvexityConstant, \smoothnessConstant > 0$, with $\strongConvexityConstant \leq \smoothnessConstant$, such that $\obj(\cdot, \param)$ is $\smoothnessConstant$-smooth and $\strongConvexityConstant$-strongly convex, 
i.e., for each $\param \in \paramSet$ 
and $\var, \var' \in \R^\varDim$:
\begin{align} \nonumber
    \obj(\var', \param) &\geq \obj(\var, \param) + \nabla \obj(\var, \param)^\top (\var' - \var) + \textstyle\frac{1}{2}\strongConvexityConstant \Vert \var' - \var \Vert_2^2, \\ \nonumber
    \obj(\var', \param) &\leq \obj(\var, \param) + \nabla \obj(\var, \param)^\top (\var' - \var) + \textstyle\frac{1}{2}\smoothnessConstant \Vert \var' - \var \Vert_2^2.
\end{align}
\end{assumption}

\begin{assumption}
\label{Assump: Constraint Set is Convex and Closed}
$\constraintSet$ is convex and closed in $\R^\varDim$.
\end{assumption}

\begin{assumption}
\label{Assump: Parameter Set is Convex and Compact}
$\paramSet$ is convex and compact in $\R^\paramDim$. 
\end{assumption}

Under Assumption \ref{Assump: Objective Properties}, for each $\param \in \paramSet$, $\obj(\cdot, \param)$ has a unique minimizer over $\constraintSet$, denoted $\varOpt(\param)$ below \cite[Thm. 2.8]{WrightRecht2022OptimizationForDataAnalysis}.

\textbf{Problem Statement}:
We aim to construct an outer approximation of the minimizer set $\varOptSet$, defined by:
\begin{align}
\label{Eqn: Set of All Parameterized Minimizers}
    \varOptSet &:= \textstyle\bigcup_{\param \in \paramSet} \arg\min_\var \obj(\var, \param) = \big\{\varOpt(\param): \param \in \paramSet \big\}.
\end{align}

%% file: 4_Preliminaries.tex
\section{Preliminaries on PGD and SLS}
\label{sec: Preliminaries on Projected Gradient Descent (PGD) and System Level Synthesis (SLS)}

We introduce fundamental properties of the PGD dynamics (Sec. \ref{subsec: Iterative Gradient-Based Descent Methods}) and SLS output feedback error controller design.

\subsection{Iterative Gradient-Based Descent Methods}
\label{subsec: Iterative Gradient-Based Descent Methods}

We first review the convergence properties of
PGD dynamics on strongly convex problems.
Consider
$f_{\text{PGD}}: \R^\varDim \times \R^\paramDim \times \R \ra \R^\varDim$ defined below.
Let $(\ratePGD_k: k \geq 0)$ denote a steplength sequence, and $\proj_\constraintSet$ be the Euclidean projection onto $\constraintSet$,
i.e., $\proj_\constraintSet(\var) := \arg\min_{\var' \in \constraintSet} \Vert \var' - \var \Vert_2$ for each $\var \in \R^\varDim$\footnote{
As proved in \cite[Ch. 7]{WrightRecht2022OptimizationForDataAnalysis}, under Assumption \ref{Assump: Constraint Set is Convex and Closed}, the projection $\proj_\constraintSet(\var)$ of any $\var \in \R^\varDim$ onto $\constraintSet$ is a unique, well-defined element of $\constraintSet$.}.
Then, for any fixed parameter $\param \in \paramSet$
the PGD updates are:
{
\setlength{\abovedisplayskip}{3.5pt}
\setlength{\belowdisplayskip}{3.5pt}
\begin{align} 
\label{Eqn: PGD Updates Notation}
    \var_{k+1} &= \dynamicsPGD(\var_k, \param, \ratePGD_k) \\ \label{Eqn: Projected Gradient Descent (PGD)}
    &:= \proj_\constraintSet\left( \var_k - \ratePGD_k \nabla_\var \obj(\var_k, \param) \right).
\end{align}
}

A key result we will leverage from optimization theory is that, under Assumptions \ref{Assump: Objective Properties} and \ref{Assump: Constraint Set is Convex and Closed}, for any fixed $\param \in \paramSet$, the iterative PGD dynamics \eqref{Eqn: PGD Updates Notation} on $\obj$
converge exponentially to the global minimizer of $\obj$ over $\constraintSet$ \cite[Sec. 7.3.3]{WrightRecht2022OptimizationForDataAnalysis}.
Below, in Prop. \ref{Prop: Convergence of PGD}, we strengthen this result under our stated assumptions, 
by showing that the PGD updates \eqref{Eqn: Projected Gradient Descent (PGD)} converge at a rate 
agnostic to the specific parameter value
$\param$ in $\paramSet$.


\begin{proposition}[\textbf{Boundedness of $\varOptSet$}]
\label{Prop: Minimizer Set is Bounded}
Under Assumption \ref{Assump: Objective Properties} and \ref{Assump: Parameter Set is Convex and Compact}, $\varOptSet$ is bounded.
\end{proposition}
Prop~\ref{Prop: Minimizer Set is Bounded} is well-known in optimization theory \cite[Sec. 2]{WrightRecht2022OptimizationForDataAnalysis}.


\begin{assumption}
\label{Assump: PGD Dynamics, Initial Iterate}
There is a known, bounded set $\varInitialSet \subseteq \constraintSet$ in which the initial iterate $\var_0$ of the PGD dynamics \eqref{Eqn: Projected Gradient Descent (PGD)} is fixed.
\end{assumption}

For any $k \geq 0$,
given $\var_0 \in \varInitialSet$, $\param \in \paramSet$, and $\ratePGD_{0:k-1} := (\ratePGD_0, \cdots, \ratePGD_{k-1}) \in \R^k$, let $\varMap_k(\var_0, \param, \ratePGD_{0:k-1})$ be the $k$-th iterate 
of the PGD dynamics \eqref{Eqn: Projected Gradient Descent (PGD)}.
For any $k \geq 0$, given $\var_0 \in \varInitialSet$ and $\ratePGD_{0:k-1} \in \R^k$, we define 
the forward reachable set for the PGD dynamics \eqref{Eqn: Projected Gradient Descent (PGD)} at iteration $k$ by:
\begin{align}
\label{Eqn: Forward Reachable Set, def}
    \varForwardReachableSet{k}(\var_0, \ratePGD_{0:k-1}) &:= \{ \varMap_k(\var_0, \param, \ratePGD_{0:k-1}): \param \in \paramSet \}.
\end{align}

\begin{proposition}
[\textbf{Uniform Convergence of PGD Dynamics}]
\label{Prop: Convergence of PGD}
Suppose Assumptions \ref{Assump: Objective Properties}, \ref{Assump: Constraint Set is Convex and Closed}, and \ref{Assump: Parameter Set is Convex and Compact} hold. Fix $\var_0 \in \varInitialSet$, let $\ratePGDMin, \ratePGDMax \in (0, \frac{2}{\smoothnessConstant})$ be given such that $\ratePGDMin \leq \ratePGDMax$, and define:
\begin{align}
    \label{eq:PGD_convergence_rate}
    \textstyle\convergenceRatePGD := \max_{\ratePGD \in [\ratePGDMin, \ratePGDMax]} \big\{\max\{|1 - \ratePGD \strongConvexityConstant|, |1 - \ratePGD \smoothnessConstant| \} \big\} < 1.
\end{align}
Fix $\ratePGD_k \in [\ratePGDMin, \ratePGDMax]$, $k \geq 0$.
Then, 
for any $\param \in \paramSet$, $k \in \N$:
\begin{align}
    \label{Eqn: Exponential Bound for PGD Convergence}
    \Vert \varMap_k(\var_0, \param, \ratePGD_{0:k-1}) - \varOpt(\param) \Vert_2 &\leq \convergenceRatePGD^k \cdot \Vert \var_0 - \varOpt(\param) \Vert_2. 
\end{align}
\end{proposition}

\begin{proof}
Below, we fix $\var_0 \in \varInitialSet$, $\param \in \paramSet$, and $\ratePGD_{0:k-1} \in [\ratePGDMin, \ratePGDMax]$, and write $\var_j = \varMap_j(\var_0, \param, \ratePGD_{0:j-1})$ for $j \in [0, k]$.
From \cite[Prop. 7.8]{WrightRecht2022OptimizationForDataAnalysis}, $\forall \ \param \in \paramSet$ and $\ratePGD_k \geq 0$, $\varOpt(\param)$ is the unique fixed point of the PGD updates \eqref{Eqn: Projected Gradient Descent (PGD)}. Thus, for any $\param \in \paramSet$ and $j \geq 0$:
\begin{align} \nonumber
    \Vert \var_{j+1} -  &\varOpt(\param) \Vert_2  = \big\Vert \proj_\constraintSet\big(\var_j - \ratePGD_j \nabla \obj(\var_j, \param) \big) \\ \nonumber
    &\hspace{5mm} - \proj_\constraintSet\big(\varOpt(\param) - \ratePGD_j \nabla \obj\big(\varOpt(\param), \param \big) \big) \big\Vert_2 \\
    \label{Eqn: Apply Recht and Wright, for Exponential Bound for PGD convergence}
    \leq &\max\big\{ |1 - \ratePGD_j \strongConvexityConstant|, |1 - \ratePGD_j \smoothnessConstant| \big\} \cdot \Vert \var_j - \varOpt(\param) \Vert_2 \\ \nonumber
    \leq &\convergenceRatePGD \cdot \Vert \var_j - \varOpt(\param) \Vert_2,
\end{align}
where \eqref{Eqn: Apply Recht and Wright, for Exponential Bound for PGD convergence} follows from \cite[Eqn. (7.8)]{WrightRecht2022OptimizationForDataAnalysis}. Recursively applying the above for $j \in [0, k-1]$, we obtain \eqref{Eqn: Exponential Bound for PGD Convergence}, as desired.
\end{proof}

Next, we use Prop. \ref{Prop: Convergence of PGD} to recast the problem of over-approximating $\varOptSet$ as the computation of forward reachable sets for PGD \eqref{Eqn: PGD Updates Notation}.
The following theorem asserts that $\varOptSet$ can be over-approximated by fixing any $\var_0 \in \varInitialSet$ and $\ratePGD_{0:k-1} \in [\ratePGDMin, \ratePGDMax]^k$, computing $\varForwardReachableSet{k}(\var_0, \ratePGD_{0:k-1})$, and inflating $\varForwardReachableSet{k}(\var_0, \ratePGD_{0:k-1})$ (in the 2-norm) by the $\convergenceRatePGD$-exponentially shrinking bound furnished by Prop. \ref{Prop: Convergence of PGD}.

\begin{theorem}
(\textbf{Over-Approximating $\varOptSet$ using the Forward Reachable Sets of PGD Dynamics})
\label{Thm: Over-approximating Minimizer Set using Forward Reachable Sets}
Fix $\var_0 \in \varInitialSet$ and $\ratePGD_{0:k-1} \in [\ratePGDMin, \ratePGDMax]^k$ arbitrarily, and define $\convergenceRadius_k(\var_0) := \convergenceRatePGD^k \cdot \max_{\param \in \paramSet} \Vert \varOpt(\param) - \var_0 \Vert_2$ for each $k \geq 0$. Then:
\begin{align}
    \label{Eqn: Forward Reachable Sets approach Minimizer Sets}
    \dist\big(\varOptSet, \varForwardReachableSet{k}(\var_0, \ratePGD_{0:k-1}) \big) \leq \convergenceRadius_k(\var_0), \hspace{5mm} \forall \ k \geq 0.
\end{align}
\end{theorem}

\begin{proof}
Fix $\var_0 \in \varInitialSet$ and $\ratePGD_{0:k-1} \in [\ratePGDMin, \ratePGDMax]^k$.
By maximizing the right-hand side of \eqref{Eqn: Exponential Bound for PGD Convergence} over all $\param \in \paramSet$, we obtain that $\Vert \varMap_k(\var_0, \param, \ratePGD_{0:k-1}) - \varOpt(\param) \Vert_2 \leq \convergenceRadius_k(\var_0)$ for each $\param \in \paramSet$, from which 
\eqref{Eqn: Forward Reachable Sets approach Minimizer Sets}
readily follows.  
\end{proof}

\begin{remark}
\label{Remark: Upper bound for convergence radius}
The upper bound $\convergenceRadius_k(\var_0)$, though $\var_0$-dependent, is continuous in $\var_0$ and thus uniformly bounded over the bounded set $\varInitialSet$ under Assumption \ref{Assump: PGD Dynamics, Initial Iterate}.


\end{remark}


\subsection{SLS-based Output Feedback Error Controller Design}
\label{subsec: System-Level Synthesis (SLS)-based Output Feedback Error Controller Design}


To over-approximate
$\varOptSet$
via Thm. \ref{Thm: Over-approximating Minimizer Set using Forward Reachable Sets}, we leverage the SLS approach in~\cite{Leeman2025RobustNonlinearOptimalControlviaSLS, leeman2026vision} to compute a forward reachable set for the PGD dynamics \eqref{Eqn: Projected Gradient Descent (PGD)}, with uncertainty over 
$\param \in \paramSet$ modeled as initial state uncertainty. 
Although many methods exist for computing reachable sets~\cite{Bansal2017HJBReachability, immrax}, we adopt the 
SLS-based methods
in~\cite{Leeman2025RobustNonlinearOptimalControlviaSLS, leeman2026vision} due to their tractability for high-dimensional system dynamics \cite{Leeman2024FastSLS}. 

We consider the following dynamics over a horizon $[0, \iterationHorizon]$, with variable-parameter tuples $\state := (\var, \param)$ as states, PGD steplengths $\ratePGD$ as control inputs, and variables $\var$ as outputs:
\begin{subequations}
\label{Eqn: True Dynamics}
\begin{alignat}{2} \label{Eqn: True Dynamics, Variable}
    \var_{k+1} &= \dynamicsPGD(\var_k, \param_k, \ratePGD_k),
    \hspace{5mm} &&\forall \ k \in [0, \iterationHorizon - 1], \\ \label{Eqn: True Dynamics, Parameter}
    \param_{k+1} &= \param_k, \hspace{5mm} &&\forall \ k \in [0, \iterationHorizon - 1], \\
    \outputs_k &= \var_k, \hspace{5mm} &&\forall \ k \in [0, \iterationHorizon - 1].
\end{alignat}
\end{subequations}

For notational ease, we define $\dynamicsFullState: \R^{\varDim} \times \R^\paramDim \times \R \ra \R^{\varDim + \paramDim}$ by 
$\dynamicsFullState(\var, \param, \ratePGD) := \big(\dynamicsPGD(\var, \param, \ratePGD), \param \big)$, simplifying \eqref{Eqn: True Dynamics, Variable}-\eqref{Eqn: True Dynamics, Parameter} to:
\begin{align}
    \label{Eqn: True Dynamics, Full State}
    \state_{k+1} &= \dynamicsFullState(\state_k, \ratePGD_k), \hspace{5mm} \forall \ k \in [0, \iterationHorizon - 1].
\end{align}
We also write $\stateStacked := (\state_0, \cdots, \state_\iterationHorizon) \in \R^{(\varDim + \paramDim)(\iterationHorizon + 1)}$, and for any $k_1, k_2 \in [0, \iterationHorizon]$ with $k_1 \leq k_2$, we write $\stateStacked_{k_1:k_2} := (\state_{k_1}, \cdots, \state_{k_2}) \in \R^{(\varDim + \paramDim)(k_2 - k_1 + 1)}$.

To satisfy Prop. \ref{Prop: Convergence of PGD} and ensure the convergence of PGD dynamics, we enforce that 
$(\ratePGD_k: k \in [0, \iterationHorizon])$ satisfy:
\begin{align} 
\label{Eqn: Steplength bound constraints}
    \ratePGDMin \leq \ratePGD_k \leq \ratePGDMax, \hspace{5mm} \forall \ k \in [0, \iterationHorizon-1].
\end{align}
Define $\ratePGDBoundSlope_1 := (\zero_{\varDim + \paramDim}, 1), \ratePGDBoundSlope_2 := (\zero_{\varDim + \paramDim}, -1) \in \R^{\varDim + \paramDim + 1}$ and $\ratePGDBoundOffset_1 := - \ratePGDMax, \ratePGDBoundOffset_2 := \ratePGDMin$. 
We can then rewrite \eqref{Eqn: Steplength bound constraints} as $\ratePGDBoundSlope_p^\top (\state_k, \ratePGD_k) + \ratePGDBoundOffset_p \leq 0$ for each $p \in [2]$.

\begin{remark}
Additional affine constraints of the form \say{$\ratePGDBoundSlope_p^\top (\state_k, \ratePGD_k) + \ratePGDBoundOffset_p \leq 0$} 
may
be enforced
to bound the distance between each variable iterate $\var_k$ and the minimizer set $\varOptSet$.
\end{remark}

To construct reachable sets for the PGD dynamics $\dynamicsFullState$ in \eqref{Eqn: True Dynamics} over $\param \in \paramSet$ with low conservativeness,
at each iteration $k \in [0, N]$, we aim to select appropriate steplengths $\ratePGD_k \in [\ratePGDMin, \ratePGDMax]$ across the iteration horizon $[0, \iterationHorizon]$ as functions of past output realizations 
$\varStacked_{0:k}$, i.e., design a causal feedback law $\feedbackMap_k: \R^{\varDim(k+1)} \ra \R$, and set 
$\ratePGD_k := \feedbackMap_k(\stateStacked_{0:k})$.
Although $\varOptSet$ can also be outer-approximated via computing forward reachable sets for the PGD dynamics with \textit{fixed} steplengths,
doing so may result in more conservative outer approximations of $\varOptSet$.

\looseness-1Since optimizing over all causal feedback laws is in general intractable, we mimic \cite{leeman2026vision, Leeman2025RobustNonlinearOptimalControlviaSLS},
which instead compute a \textit{nominal state-control trajectory} $(\stateNominalStacked, \ratePGDNominalStacked)$ and an \textit{affine causal output error feedback law} mitigating the deviation of the trajectory rollouts from $(\stateNominalStacked, \ratePGDNominalStacked)$.
We first reformulate the true nonlinear system \eqref{Eqn: True Dynamics}, encoding initial uncertainty over the parameter $\param_0$, 
as the sum of a nominal nonlinear system \textit{without initial state uncertainty}~\eqref{Eqn: Nominal Dynamics} and linear time-varying (LTV) error dynamics~\eqref{Eqn: State Error LTV Dynamics}.
We then design (a) state and control trajectories $\stateNominalStacked
:= (\varNominal_0, \paramNominal_0, \cdots, \varNominal_\iterationHorizon, \paramNominal_\iterationHorizon) 
\in \R^{(\varDim + \paramDim)(\iterationHorizon + 1)}$ and $\ratePGDNominalStacked := (\ratePGDNominal_0, \cdots, \ratePGDNominal_\iterationHorizon) \in \R^{\iterationHorizon + 1}$,
satisfying the nominal dynamics:
\begin{align}
\label{Eqn: Nominal Dynamics}
    \stateNominal_{k+1} &= \dynamicsFullState(\stateNominal_k, \ratePGDNominal_k), \hspace{5mm} \forall \ k \in [0, \iterationHorizon - 1].
\end{align}
and (b) a set of causal affine output error feedback gains $\setFeedbackGains := \{\feedbackGain_{k,j} \in \R^{1 \times \varDim}, \ k,j \in [0,\iterationHorizon], k \geq j \}$
parameterizing the following causal affine output error feedback map:
\begin{align} \nonumber
    \ratePGD_k &= \feedbackMap_k(\outputStacked_{0:k}) \\
    \label{Eqn: Affine Output Feedback Map}
    &= \ratePGDNominal_k + \textstyle\sum_{j=0}^k \feedbackGain_{k,j}(\var_{k-j} - \varNominal_{k-j}),
    \hspace{3mm} \forall \ k \in [0, \iterationHorizon].
\end{align}

\begin{remark}
\label{Remark: Choice of Initial Nominal Parameter, paramNominal 0}
Designing $\stateNominalStacked$ requires choosing
an initial nominal state $\varNominal_0$, which we assume to be the true initial state $\var_0$ for simplicity, and an initial nominal parameter $\paramNominal_0$, which is then fixed across iterations per \eqref{Eqn: True Dynamics, Parameter}.
In practice, 
$\paramNominal_0 \in \paramSet$ can be selected to minimize the worst-case parameter deviation over $\paramSet$, i.e., $\paramNominal_0 \in \arg\min_{\param \in  \paramSet} \max_{\param' \in \paramSet} \Vert \param - \param' \Vert_2$, .
\end{remark}


%% file: 5_Methods.tex
\section{Methods}
\label{sec: Methods}

In this section, we outer-approximate the minimizer set $\varOptSet$ using Thm. \ref{Thm: Over-approximating Minimizer Set using Forward Reachable Sets} by computing reachable sets for the PGD dynamics with SLS.
For differentiable PGD dynamics, we formulate a linear time-varying (LTV) system that describes the propagation of the error between nominal and true trajectories
across iterations (Sec. \ref{subsec: LTV Error System for Smooth PGD Dynamics}), and construct error bounds across iterations (Sec. \ref{subsec: Linearization Error Bounds}).
We then generalize our methods
to the setting in which the PGD dynamics are non-differentiable
via a smoothing technique
(Sec. \ref{subsec: Smoothing PGD Dynamics when Non-Differentiable}).


\subsection{LTV Error System for Smooth PGD Dynamics}
\label{subsec: LTV Error System for Smooth PGD Dynamics}


We
derive the LTV error system for the true dynamics \eqref{Eqn: True Dynamics},
as a step towards designing error feedback gains,
for the case in which $\constraintSet$ is $\R^\varDim$ or an affine subspace of $\R^\varDim$.
In this case, the projection operator $\proj_\constraintSet(\cdot)$ is affine, 
and thus $\dynamicsFullState$ in \eqref{Eqn: True Dynamics}
is differentiable and satisfies
Assumption \ref{Assump: Objective Properties}.
In Sec. \ref{subsec: Smoothing PGD Dynamics when Non-Differentiable}, we study the general setting where $\dynamicsFullState$ may be non-differentiable.

\looseness=-1The LTV error system admits, across iterations $k \in [0, \iterationHorizon]$, the state errors 
$\stateError_k := \state_k - \stateNominal_k$
as states and steplength errors $\ratePGDError_k := \ratePGD_k - \ratePGDNominal_k$ 
as controls. 
Its gain and actuation matrices for each $k \in [0, \iterationHorizon]$ are computed
via the Jacobian linearization of \eqref{Eqn: True Dynamics, Full State} at the nominal state-control 
pair $(\stateNominal_k, \ratePGDNominal_k)$
via: 
{
\setlength{\abovedisplayskip}{3.5pt}
\setlength{\belowdisplayskip}{3.5pt}
\begin{subequations}
\label{Eqn: Jacobian Matrices}
\begin{align}
    \label{Eqn: Jacobian Matrix, State}
    \jacobianState_k(\stateNominal_k, \ratePGDNominal_k) &:= \textstyle\frac{\partial \dynamicsFullState}{\partial \state}(\stateNominal_k, \ratePGDNominal_k), \\
    \label{Eqn: Jacobian Matrix, Control}
    \jacobianControl_k(\stateNominal_k, \ratePGDNominal_k) &:= \textstyle\frac{\partial \dynamicsFullState}{\partial \ratePGD}(\stateNominal_k, \ratePGDNominal_k),
\end{align}
\end{subequations}
}
\hspace{-1mm}and the \textit{linearization error} $\linearizationError_k(\state_k, \ratePGD_k, \stateNominal_k, \ratePGDNominal_k) \in \R^{\varDim + \paramDim}$ as:
{
\setlength{\abovedisplayskip}{3.5pt}
\setlength{\belowdisplayskip}{3.5pt}
\begin{align}
\label{Eqn: Linearization Error}
    \linearizationError_k(\state_k, \ratePGD_k, \stateNominal_k, \ratePGDNominal_k) &= \dynamicsFullState(\state_k, \ratePGD_k) - \dynamicsFullState(\stateNominal_k, \ratePGDNominal_k) \\ \nonumber 
    &\hspace{5mm} - \jacobianState_k(\stateNominal_k, \ratePGDNominal_k) \stateError_k - \jacobianControl_k(\stateNominal_k, \ratePGDNominal_k) \ratePGDError_k.
\end{align}
}
\hspace{-2mm}We then use
\eqref{Eqn: Jacobian Matrices} and \eqref{Eqn: Linearization Error} to describe the propagation of parameter errors as the following LTV system
while expressing
each linearization error as a disturbance\footnote{
In Sec. \ref{subsec: Smoothing PGD Dynamics when Non-Differentiable}, which considers the setting of non-differentiable PGD dynamics, the disturbance $\disturbance_k$ will also incorporate the approximation error between the PGD dynamics and a differentiable proxy.
}
$\disturbance_k$:
{
\setlength{\abovedisplayskip}{3.5pt}
\setlength{\belowdisplayskip}{3.5pt}
\begin{align} \label{Eqn: Disturbance}
    \disturbance_k &= \linearizationError_k(\state_k, \ratePGD_k, \stateNominal_k, \ratePGDNominal_k), \\
    \nonumber
    \stateError_{k+1} &:= \state_{k+1} - \stateNominal_{k+1} = \dynamicsFullState(\state_k, \ratePGD_k) - \dynamicsFullState(\stateNominal_k, \ratePGDNominal_k) \\ \label{Eqn: State Error LTV Dynamics}
    &= \jacobianState_k(\stateNominal_k, \ratePGDNominal_k) \stateError_k + \jacobianControl_k(\stateNominal_k, \ratePGDNominal_k) \ratePGDError_k + \disturbance_k,
\end{align}
}
\hspace{-1.5mm}Meanwhile, we define the \textit{control error} $\ratePGDError_k := \ratePGD_k - \ratePGDNominal_k$ for each $k \in [0, \iterationHorizon]$ and the \textit{output matrix} $\outputMatrix := \begin{bmatrix}
    I_\varDim & O_{\varDim + \paramDim}
\end{bmatrix} \in \R^{\varDim \times (\varDim + \paramDim)}$, and derive from \eqref{Eqn: Affine Output Feedback Map} that:
{
\setlength{\abovedisplayskip}{3.5pt}
\setlength{\belowdisplayskip}{3.5pt}
\begin{align}
    \label{Eqn: Control Error as affine function of State Errors}
    \ratePGDError_k 
    &= \textstyle\sum_{j=0}^k \feedbackGain_{k,j} \outputMatrix \stateError_{k-j},
    \hspace{2mm} \forall \ k \in [0, \iterationHorizon].
\end{align}
}


A key insight of the SLS framework is that the design of the output feedback gains $\setFeedbackGains$ can be recast as the design of a \textit{system response}, mapping disturbances $\disturbanceStacked$ directly to state and control errors.
For brevity, we stack the LTV system dynamics \eqref{Eqn: State Error LTV Dynamics} and \eqref{Eqn: Control Error as affine function of State Errors} across iterations $k \in [0, \iterationHorizon]$. 
Let $\disturbanceStacked := (\stateError_0, \disturbance_0, \cdots, \disturbance_{\iterationHorizon-1})$, and define 
$\jacobianStateStacked$, $\jacobianControlStacked$, $\outputMatrixStacked$, $\feedbackGainStacked$, and $\downshift$ by:
\footnote{For notational simplicity, the explicit dependence of each $\jacobianState$ and $\jacobianControl$ on the nominal state-control pair $(\state_k, \ratePGD_k)$ is omitted from mention in \eqref{Eqn: Stacking State Jacobians}-\eqref{Eqn: Stacking Control Jacobians}.
}
{
\setlength{\abovedisplayskip}{1pt}
\setlength{\belowdisplayskip}{3.5pt}
\begin{subequations}
\label{Eqn: Stacking Matrices}
\begin{align}
    \label{Eqn: Stacking State Jacobians}
    \jacobianStateStacked &:= \blockDiag\{ \jacobianState_0, \cdots, \jacobianState_{\iterationHorizon-1}, O_{\varDim + \paramDim} \}, \\
    \label{Eqn: Stacking Control Jacobians}
    \jacobianControlStacked &:= \blockDiag\{ \jacobianControl_0, \cdots, \jacobianControl_{\iterationHorizon-1}, \zero_{\varDim + \paramDim} \}, \\
    \label{Eqn: Stacking Output Matrices}
    \outputMatrixStacked &:= \blockDiag\{ \outputMatrix, \cdots, \outputMatrix, \outputMatrix \}, \\
    \label{Eqn: Stacking Feedback Gains}
    \feedbackGainStacked &:= \begin{bmatrix}
        \feedbackGain_{0, 0} & O & \cdots & O \\[-7pt]
        \feedbackGain_{1, 0} & \feedbackGain_{1, 1} & \ddots & O \\[-3pt]
        \vdots & \vdots & \ddots & \vdots \\[-7pt]
        \feedbackGain_{\iterationHorizon, 0} & \feedbackGain_{\iterationHorizon, 1} & \ddots & \feedbackGain_{\iterationHorizon, \iterationHorizon}
    \end{bmatrix}, \\
    \label{Eqn: Downshift Operator}
    \downshift &:= \begin{bmatrix}
        O_{(\varDim + \paramDim) \times (\varDim + \paramDim)\iterationHorizon} & O_{\varDim + \paramDim} \\
        I_{(\varDim + \paramDim)\iterationHorizon} & O_{(\varDim + \paramDim)\iterationHorizon \times (\varDim + \paramDim)} 
    \end{bmatrix}
\end{align}
\end{subequations}
}

We use \eqref{Eqn: Stacking Matrices} to rewrite \eqref{Eqn: State Error LTV Dynamics}-\eqref{Eqn: Control Error as affine function of State Errors} across $k \in [0, \iterationHorizon]$ as:%
{
\setlength{\abovedisplayskip}{3.5pt}
\setlength{\belowdisplayskip}{3.5pt}
\begin{subequations}
\label{Eqn: State and Control Errors, Stacked}
\begin{align}
    \label{Eqn: State Errors, Stacked}
    \stateErrorStacked &= \downshift \jacobianStateStacked  \stateErrorStacked + \downshift 
    \jacobianControlStacked \ratePGDErrorStacked + \disturbanceStacked, \\
    \label{Eqn: Control Errors, Stacked}
    \ratePGDErrorStacked &= \feedbackGainStacked \outputMatrixStacked \stateErrorStacked,
\end{align}
\end{subequations}
}
\hspace{-1mm}which can be further written as:
{
\setlength{\abovedisplayskip}{4pt}
\setlength{\belowdisplayskip}{4pt}
\begin{align}
    \label{Eqn: State and Control Errors, Stacked, written using System Responses}
    \begin{bmatrix}
        \stateErrorStacked \\ \ratePGDErrorStacked
    \end{bmatrix}
    &= \begin{bmatrix}
        \systemResponse_{\state \disturbance} & \systemResponse_{\state \outputNoise} \\
        \systemResponse_{\ratePGD \disturbance} & \systemResponse_{\ratePGD \outputNoise}
    \end{bmatrix}
    \begin{bmatrix}
        \disturbanceStacked \\ \zero_{\varDim}
    \end{bmatrix} := \systemResponse \begin{bmatrix}
        \disturbanceStacked \\ \zero_{\varDim}
    \end{bmatrix},
\end{align}
}
\hspace{-1.5mm}where the constituent blocks of the \textit{system response} $\systemResponse \in \R^{(\varDim + \paramDim + 1)(\iterationHorizon+1) \times (2\varDim + \paramDim)(\iterationHorizon+1)}$ are:
{
\setlength{\abovedisplayskip}{3.5pt}
\setlength{\belowdisplayskip}{3.5pt}
\begin{subequations} \label{Eqn: System Responses from Output Feedback Gains}
\begin{align} 
    \label{Eqn: System Response from Feedback Gains (State, Disturbance)}
    \systemResponse_{\state \disturbance} &:= (I_{(\varDim + \paramDim)\iterationHorizon} - \downshift \jacobianStateStacked - \downshift \jacobianControlStacked \feedbackGainStacked \outputMatrixStacked)^{-1} \\ \label{Eqn: System Response from Feedback Gains (State, Output Noise)}
    \systemResponse_{\state \outputNoise} &:= (I_{(\varDim + \paramDim)\iterationHorizon} - \downshift \jacobianStateStacked - \downshift \jacobianControlStacked \feedbackGainStacked \outputMatrixStacked)^{-1} \downshift \jacobianControlStacked \feedbackGainStacked \\ 
    \label{Eqn: System Response from Feedback Gains (Input, Disturbance)}
    \systemResponse_{\ratePGD \disturbance} &:= \feedbackGainStacked \outputMatrixStacked (I - \downshift \jacobianStateStacked - \downshift \jacobianControlStacked \feedbackGainStacked \outputMatrixStacked)^{-1} \\ 
    \label{Eqn: System Response from Feedback Gains (Input, Output Noise)}
    \systemResponse_{\ratePGD \outputNoise} &:= \feedbackGainStacked \outputMatrixStacked (I - \downshift \jacobianStateStacked - \downshift \jacobianControlStacked \feedbackGainStacked \outputMatrixStacked)^{-1} \downshift \jacobianControlStacked \feedbackGainStacked + \feedbackGainStacked.
\end{align}
\end{subequations}
}
\hspace{-1mm}Conversely, from \eqref{Eqn: System Responses from Output Feedback Gains}, we also obtain that:
{
\setlength{\abovedisplayskip}{3.5pt}
\setlength{\belowdisplayskip}{3.5pt}
\begin{align} \label{Eqn: Feedback Error Gains from System Response}
    \feedbackGainStacked &= \systemResponse_{\ratePGD \outputNoise} - \systemResponse_{\ratePGD \disturbance} \systemResponse_{\state \disturbance}^{-1} \systemResponse_{\state \outputNoise}.
\end{align}
}

As proved in \cite{anderson2019levelsynthesis, ZhouTzoumas2023SafeControlofPartiallyObservedLTVSystems} and described below, the system response $\systemResponse$ parameterizes all causal output feedback gains $\feedbackGainStacked$.

\begin{proposition}\cite[Prop. 1]{ZhouTzoumas2023SafeControlofPartiallyObservedLTVSystems}
If there exist output feedback gains $\feedbackGainStacked$ of the form \eqref{Eqn: Stacking Feedback Gains} satisfying \eqref{Eqn: State and Control Errors, Stacked}, then $\systemResponse_{\state \disturbance}$, $ \systemResponse_{\state \outputNoise}$, $ \systemResponse_{\ratePGD \disturbance}$, and $\systemResponse_{\ratePGD \outputNoise}$, as computed by \eqref{Eqn: System Responses from Output Feedback Gains},
satisfy: 
{
\setlength{\abovedisplayskip}{3.5pt}
\setlength{\belowdisplayskip}{3.5pt}
\begin{subequations} \label{Eqn: System Responses, Affine Constraints}
\begin{align} \label{Eqn: System Responses, Affine Constraints, Controllability}
    \begin{bmatrix}
        I - \downshift \A & - \downshift \B
    \end{bmatrix} \systemResponse &= 
    \begin{bmatrix}
        I & O
    \end{bmatrix}, \\ \label{Eqn: System Responses, Affine Constraints, Observability}
    \systemResponse \begin{bmatrix}
        I - \downshift \jacobianStateStacked \\ -\outputMatrixStacked
    \end{bmatrix} &= 
    \begin{bmatrix}
        I \\ O
    \end{bmatrix}, \\ \label{Eqn: System Responses, Affine Constraints, Upper Triangular}
    \systemResponse_{\state \disturbance}, \systemResponse_{\state \outputNoise}, \systemResponse_{\ratePGD \disturbance}, \systemResponse_{\ratePGD \outputNoise} &\text{ are lower block-triangular}.
\end{align}
\end{subequations}
}
\hspace{-1mm}Conversely, if $\systemResponse_{\state \disturbance}$, $ \systemResponse_{\state \outputNoise}$, $ \systemResponse_{\ratePGD \disturbance}$, and $\systemResponse_{\ratePGD \outputNoise}$ 
satisfy \eqref{Eqn: System Responses, Affine Constraints}, then $\feedbackGainStacked$, as computed by \eqref{Eqn: Feedback Error Gains from System Response}, is lower block triangular.
\end{proposition}

Below, given $\systemResponse$ satisfying \eqref{Eqn: System Responses, Affine Constraints}, for each $k \in [0, \iterationHorizon]$ and $j \in [0, k]$, we respectively define $(\systemResponse_{\state \disturbance})_{k,j}$ and $(\systemResponse_{\ratePGD \disturbance})_{k,j}$ to be the $(k, j)$-th block
of $\systemResponse_{\state \disturbance}$ and $\systemResponse_{\ratePGD \disturbance}$ in the sense of \eqref{Eqn: Stacking Feedback Gains}.



\subsection{Linearization Error Bounds}
\label{subsec: Linearization Error Bounds}



To compute forward reachable sets for the dynamics~\eqref{Eqn: True Dynamics} while enforcing the steplength constraints \eqref{Eqn: Steplength bound constraints}, we first bound the disturbance \eqref{Eqn: Disturbance} to characterize the error accrued from linearizing $\dynamicsFullState$. 
In Prop. \ref{Prop: Lipschitz Continuity of Gradient of PGD Dynamics},
we
begin by bounding the curvature of $\dynamicsFullState$, 
in Prop. \ref{Prop: Lipschitz Continuity of Gradient of PGD Dynamics}, 
using the Lipschitz constant of its gradient over all possible iterate rollouts.
To this end, define:
{
\setlength{\abovedisplayskip}{3.5pt}
\setlength{\belowdisplayskip}{3.5pt}
\begin{align} \label{Eqn: Convex Hull of all Trajectories}
    \varConvexHull &:= \conv\big\{ \varForwardReachableSet{k}(\var_0, \ratePGD_{0:k-1}): \\ \nonumber
    &\hspace{2cm} k \geq 0, \var_0 \in \varInitialSet, \ratePGD_{0:k-1} \in [\ratePGDMin, \ratePGDMax]^k \big\}
\end{align}
}
\hspace{-1.7mm}as the convex hull of all possible rollouts of the PGD dynamics \eqref{Eqn: Projected Gradient Descent (PGD)} from any $\var_0 \in \varInitialSet$.
We prove that, for a suitable $\constraintSet$, the gradient of \eqref{Eqn: True Dynamics} is Lipschitz continuous over $\varConvexHull \times \paramSet$.

\begin{proposition}
(\textbf{$\dynamicsFullState$ Admits Lipschitz Gradients})
\label{Prop: Lipschitz Continuity of Gradient of PGD Dynamics}
Suppose Assumptions \ref{Assump: Objective Properties}, \ref{Assump: Parameter Set is Convex and Compact}, and \ref{Assump: PGD Dynamics, Initial Iterate} hold, and $\constraintSet$ is $\R^\varDim$ or an affine subspace of $\R^\varDim$.
For each $i \in [\varDim + \paramDim]$, let $\dynamicsFullState_i: \R^{\varDim + \paramDim + 1} \ra \R$ denote the $i$-th component of the dynamics $\dynamicsFullState$ in \eqref{Eqn: True Dynamics}. 
Then, $\forall \ i \in [\varDim + \paramDim]$, there exists some $\hessianBound_i > 0$ such that, $\forall \ (\state, \ratePGD), (\state', \ratePGD') \in \varConvexHull \times \paramSet \times [\ratePGDMin, \ratePGDMax]$:
\begin{align}
    \label{Eqn: Lipschitz Continuity of Gradient of PGD Dynamics}
    &\Vert \nabla \dynamicsFullState_i(\state, \ratePGD) - \nabla \dynamicsFullState_i(\state', \ratePGD') \Vert_1 \hspace{-0.4mm} \leq \hspace{-0.4mm} \hessianBound_i \Vert (\state, \ratePGD) - (\state', \ratePGD') \Vert_\infty.
\end{align}
\end{proposition}

\begin{proof}
We first verify that $\varConvexHull$ is bounded.
From \eqref{Eqn: Forward Reachable Sets approach Minimizer Sets} in Thm. \ref{Thm: Over-approximating Minimizer Set using Forward Reachable Sets},
$\forall \ \var_0 \in \varInitialSet$, $\param \in \paramSet$, and $\ratePGD_{0:k-1} \in [\ratePGDMin, \ratePGDMax]^k$, we have 
$\Vert \varMap_k(\var_0, \param, \ratePGD_{0:k-1}) - \varOpt(\param) \Vert_2 \leq \convergenceRadius_k(\var_0)$ as $k \ra \infty$. Maximizing
over $\param \in \paramSet$, we find $\dist(\varForwardReachableSet{k}(\var_0, \ratePGD_{0:k-1}), \varOptSet) \leq \convergenceRadius_k(\var_0)$ as $k \ra \infty$.
Since $\varForwardReachableSet{0}(\var_0, \param) = \varInitialSet$ and $\varOpt$ are bounded (by Assumption \ref{Assump: PGD Dynamics, Initial Iterate} and Prop. \ref{Prop: Minimizer Set is Bounded}, resp.), and $\convergenceRadius_k(\var_0) \geq 0$ decreases 
to 0 as $k \ra \infty$, we conclude that the sets $\varForwardReachableSet{k}(\var_0, \ratePGD_{0:k-1})$ are uniformly bounded across all $k \geq 0$, $\var_0 \in \varInitialSet$, and $\ratePGD_{0:k-1} \in [\ratePGDMin, \ratePGDMax]^k$, and thus $\varConvexHull$ is bounded.

Next, since $\constraintSet$ is $\R^\varDim$ or an affine subspace of $\R^\varDim$, the projection map $\proj_\constraintSet$ is affine. 
Thus, from \eqref{Eqn: Projected Gradient Descent (PGD)} and \eqref{Eqn: True Dynamics}, 
each $\nabla \dynamicsFullState_i$
is bilinear in $\ratePGD$ and the first-order and second-order partial derivatives of $\obj$.
For $\nabla \dynamicsFullState_i$ 
to be Lipschitz continuous over $\varConvexHull \times \paramSet \times [\ratePGDMin, \ratePGDMax]$, it thus suffices to verify that the second-order partial derivatives of $\obj$ are Lipschitz continuous and bounded over $\varConvexHull \times \paramSet$, which follows from Assumption \ref{Assump: Objective Properties} and \ref{Assump: Parameter Set is Convex and Compact} and the boundedness of $\varConvexHull$ established above.
\end{proof}

Equipped with curvature bounds of $\dynamicsFullState$ provided by Prop. \ref{Prop: Lipschitz Continuity of Gradient of PGD Dynamics}, we now bound the disturbance $\disturbance_k = \linearizationError(\state_k, \ratePGD_k, \stateNominal_k, \ratePGDNominal_k)$ in \eqref{Eqn: Disturbance}, in a manner analogous to \cite[Prop. III.1]{Leeman2025RobustNonlinearOptimalControlviaSLS}.

\begin{proposition}[\textbf{Linearization Error Bound}]
\label{Prop: Quadratic Bound on Linearization Error}
\looseness-1Suppose Assumptions \ref{Assump: Objective Properties}, \ref{Assump: Parameter Set is Convex and Compact}, and \ref{Assump: PGD Dynamics, Initial Iterate} hold, and $\constraintSet$ is $\R^\varDim$ or an affine subspace of $\R^\varDim$.
For each $k \in [0, \iterationHorizon]$ and $i \in [\varDim + \paramDim]$, we define $r_{k,i}: \R^{2(\varDim + \paramDim)} \ra \R$ to be the $i$-th coordinate of the linearization error map $r_k$,
for each $i \in [\varDim + \paramDim]$.
Then for any $\var_0 \in  \varInitialSet$, $\ratePGD_{0:k-1}, \ratePGDNominal_{0:k-1} \in [\ratePGDMin, \ratePGDMax]^k$, any $\state_k = (\var_k, \param_k) \in \varForwardReachableSet{k}(\var_0, \ratePGD_{0:k-1}) \times \paramSet$ and $\stateNominal_k = (\varNominal_k, \paramNominal_k) \in \varForwardReachableSet{k}(\var_0, \ratePGDNominal_{0:k-1}) \times \paramSet$, and any $i \in [\varDim + \paramDim]$, we have:
\begin{align}
\label{Eqn: Quadratic Bound on Linearization Error}
    |\linearizationError_{k,i}(\state_k, \ratePGD_k, \stateNominal_k, \ratePGDNominal_k)| &\leq \hessianBound_i \cdot \Vert (\stateError_k, \ratePGDError_k) \Vert_\infty^2.
\end{align}
\end{proposition}

\begin{proof}
\looseness-1 By the Mean Value Theorem, there exists a $\lambda_i \in [0, 1]$ with:
{
\setlength{\abovedisplayskip}{3.5pt}
\setlength{\belowdisplayskip}{3.5pt}
\begin{align} \nonumber
    &\dynamicsFullState_i(\state_k, \ratePGD_k) - \dynamicsFullState_i(\stateNominal_k, \ratePGDNominal_k) \\ \nonumber
    = \ &\nabla \dynamicsFullState_i\big( \lambda_i(\state_k, \ratePGD_k) + (1 - \lambda_i)(\stateNominal_k, \ratePGDNominal_k) \big)^\top (\stateError_k, \ratePGDError_k).
\end{align}
}
\hspace{-1mm}By construction, $(\state_k, \ratePGD_k), (\stateNominal_k, \ratePGDNominal_k) \in \varConvexHull \times \paramSet \times [\ratePGDMin, \ratePGDMax]$, and thus 
$\lambda_i(\state_k, \ratePGD_k) + (1 - \lambda_i)(\stateNominal_k, \ratePGDNominal_k) \in \varConvexHull \times \paramSet \times [\ratePGDMin, \ratePGDMax]$ (Note that $\varConvexHull$ and $[\ratePGDMin, \ratePGDMax]$ are convex by construction, and $\paramSet$ is convex under Assumption \ref{Assump: Parameter Set is Convex and Compact}).
Thus, we have:
{
\setlength{\abovedisplayskip}{3.5pt}
\setlength{\belowdisplayskip}{3.5pt}
\begin{align} \nonumber
    &|\linearizationError_{k,i}(\state_k, \ratePGD_k, \stateNominal_k, \ratePGDNominal_k)| \\ \nonumber
    = \ &|\dynamicsFullState_i(\state_k, \ratePGD_k) - \dynamicsFullState_i(\stateNominal_k, \ratePGDNominal_k) - \nabla \dynamicsFullState_i(\stateNominal_k, \ratePGDNominal_k)^\top (\stateError_k, \ratePGDError_k) | \\ \nonumber
    \leq \ &\Vert \nabla \dynamicsFullState_i\big( \lambda_i(\state_k, \ratePGD_k) + (1 - \lambda_i)(\stateNominal_k, \ratePGDNominal_k) \big) - \nabla \dynamicsFullState_i(\stateNominal_k, \ratePGDNominal_k) \Vert_1 \\ \nonumber
    &\hspace{1cm} \cdot \Vert (\stateError_k, \ratePGDError_k) \Vert_\infty \\ 
    \label{Eqn: Apply Lipschitz bounds of Dynamics' Gradient}
    \leq \ &\mu_i \cdot \lambda_i \Vert (\stateError_k, \ratePGDError_k) \Vert_\infty \cdot \Vert (\stateError_k, \ratePGDError_k) \Vert_\infty,
    \\ \nonumber
    \leq \ &\hessianBound_i \cdot \Vert (\stateError_k, \ratePGDError_k) \Vert_\infty^2.
\end{align}
}
where \eqref{Eqn: Apply Lipschitz bounds of Dynamics' Gradient} follows from Prop. \ref{Prop: Lipschitz Continuity of Gradient of PGD Dynamics}.
\end{proof}

We define $\hessianBoundStacked := \diag\{\hessianBound_1, \cdots, \hessianBound_{\varDim + \paramDim} \} \in \R^{(\varDim + \paramDim) \times (\varDim + \paramDim)}$, which allows us to stack \eqref{Eqn: Quadratic Bound on Linearization Error} across all $i \in [\varDim + \paramDim]$ as:
{
\setlength{\abovedisplayskip}{3.5pt}
\setlength{\belowdisplayskip}{3.5pt}
\begin{align}
\label{Eqn: Disturbance Bound, Stacked}
    \disturbance_k = \linearizationError_k(\state_k, \ratePGD_k, \stateNominal_k, \ratePGDNominal_k) \in \Vert (\stateError_k, \ratePGDError_k) \Vert_\infty^2 \hessianBoundStacked \ballUnit_\infty^{\varDim + \paramDim}
\end{align}
}

Below, in Thm. \ref{Thm: Robust Steplength Constraint Satisfaction}, we leverage Prop. \ref{Prop: Quadratic Bound on Linearization Error} to establish robust guarantees for satisfying the steplength constraint \eqref{Eqn: Steplength bound constraints}.

\begin{theorem}
[\textbf{Robust Steplength Constraint Satisfaction}]
\label{Thm: Robust Steplength Constraint Satisfaction}
Suppose Assumptions \ref{Assump: Objective Properties}, \ref{Assump: Constraint Set is Convex and Closed}, 
\ref{Assump: Parameter Set is Convex and Compact},
and \ref{Assump: PGD Dynamics, Initial Iterate} hold.
Given a nominal trajectory $(\stateNominalStacked, \ratePGDNominalStacked)$ of the dynamics $\dynamicsFullState$ in \eqref{Eqn: True Dynamics} and a system response $\systemResponse$ satisfying \eqref{Eqn: System Responses, Affine Constraints}, 
if there exist auxiliary error upper bounds $\tubeSizeStacked := (\tubeSize_0, \cdots, \tubeSize_\iterationHorizon)$ satisfying, $\forall k \in [0, \iterationHorizon]$:
{
\setlength{\abovedisplayskip}{3.5pt}
\setlength{\belowdisplayskip}{3.5pt}
\begin{align} 
\label{Eqn: Inequality for Robust Constraint Satisfaction}
    &\ratePGDBoundSlope_p^\top (\stateNominal_k, \ratePGDNominal_k) + \ratePGDBoundOffset_p + \left\Vert \ratePGDBoundSlope_p^\top
    \begin{bmatrix}
        (\systemResponse_{\state \disturbance})_{k, 0} \\
        (\systemResponse_{\ratePGD \disturbance})_{k, 0}
    \end{bmatrix}
    \right\Vert_1 \diam(\paramSet) \\ \nonumber
    &\hspace{5mm} + \sum_{j=1}^k \tubeSize_{j-1}^2 \left\Vert \ratePGDBoundSlope_p^\top
    \begin{bmatrix}
        (\systemResponse_{\state \disturbance})_{k, 0} \\
        (\systemResponse_{\ratePGD \disturbance})_{k, 0}
    \end{bmatrix}
    \hessianBoundStacked
    \right\Vert_1 \leq 0, \quad \forall \ p \in [2], \\ 
\label{Eqn: Inequality for Propagating Recursive Tube Bounds}
    &\left\Vert 
    \ratePGDBoundSlope_p^\top
    \begin{bmatrix}
        (\systemResponse_{\state \disturbance})_{k, 0} \\
        (\systemResponse_{\ratePGD \disturbance})_{k, 0}
    \end{bmatrix}
    \right\Vert_\infty \diam(\paramSet) \\ \nonumber
    &\hspace{5mm} + \sum_{j=1}^k \tubeSize_{j-1}^2 \left\Vert 
    \begin{bmatrix}
        (\systemResponse_{\state \disturbance})_{k, 0} \\
        (\systemResponse_{\ratePGD \disturbance})_{k, 0}
    \end{bmatrix}
    \hessianBoundStacked
    \right\Vert_\infty \leq \tubeSize_k,
\end{align}
}
\hspace{-1.2mm}then $\Vert (\stateError_k, \ratePGDError_k) \Vert_\infty \leq \tubeSize_k$ and $\ratePGD_k \in [\ratePGDMin, \ratePGDMax]$ for each $k \in [0, \iterationHorizon]$,
regardless of the ground truth value of $\param \in \paramSet$.
\end{theorem}

\begin{proof}
We provide a proof by induction.
At $k = 0$, we prove $\Vert (\stateError_0, \ratePGD_0) \Vert_\infty \leq \tubeSize_0$ using \eqref{Eqn: State and Control Errors, Stacked, written using System Responses}, as follows:
{
\setlength{\abovedisplayskip}{3.5pt}
\setlength{\belowdisplayskip}{3.5pt}
\begin{align} \nonumber
    \left\Vert 
    \begin{bmatrix}
        \stateError_0 \\
        \ratePGDError_0
    \end{bmatrix}
    \right\Vert_\infty &= \left\Vert 
    \begin{bmatrix}
        (\systemResponse_{\state \disturbance})_{0, 0} \\
        (\systemResponse_{\ratePGD \disturbance})_{0, 0}
    \end{bmatrix}
    \stateError_0
    \right\Vert_\infty \\
    \label{Eqn: Apply Inequality for Propagating Recursive Tube Bounds, when k is 0}
    &\leq \left\Vert 
    \begin{bmatrix}
        (\systemResponse_{\state \disturbance})_{0, 0} \\
        (\systemResponse_{\ratePGD \disturbance})_{0, 0}
    \end{bmatrix}
    \right\Vert_\infty 
    \Vert \stateError_0 \Vert_\infty
    \leq \tubeSize_0,
\end{align}
}
\hspace{-1mm}where \eqref{Eqn: Apply Inequality for Propagating Recursive Tube Bounds, when k is 0} follows by taking $k = 0$ in \eqref{Eqn: Inequality for Propagating Recursive Tube Bounds} and noting that $\Vert \stateError_0 \Vert_\infty = \Vert (0, \paramError_0) \Vert_\infty \leq \diam(\paramSet)$. 
Similarly, we verify robust constraint satisfaction at $k = 0$ as follows:
{
\setlength{\abovedisplayskip}{3.5pt}
\setlength{\belowdisplayskip}{3.5pt}
\begin{align} \nonumber
    &\ratePGDBoundSlope_p^\top (\state_0, \ratePGD_0) + \ratePGDBoundOffset_p \\ \nonumber
    = \ &\ratePGDBoundSlope_p^\top (\stateNominal_0, \ratePGDNominal_0) + \ratePGDBoundOffset_p + \ratePGDBoundSlope_p^\top (\stateError_0, \ratePGDError_0) \\ \label{Eqn: Apply Inequality for Robust Constraint Satisfaction, k = 0}
    \leq \ &\ratePGDBoundSlope_p^\top (\stateNominal_0, \ratePGDNominal_0) + \ratePGDBoundOffset_p + 
    \left\Vert 
    \ratePGDBoundSlope_p^\top
    \begin{bmatrix}
        (\systemResponse_{\state \disturbance})_{0, 0} \\
        (\systemResponse_{\ratePGD \disturbance})_{0, 0}
    \end{bmatrix}
    \stateError_0
    \right\Vert_\infty\\ 
    \leq \ &\ratePGDBoundSlope_p^\top (\stateNominal_0, \ratePGDNominal_0) + \ratePGDBoundOffset_p + 
    \left\Vert 
    \ratePGDBoundSlope_p^\top
    \begin{bmatrix}
        (\systemResponse_{\state \disturbance})_{0, 0} \\
        (\systemResponse_{\ratePGD \disturbance})_{0, 0}
    \end{bmatrix}
    \right\Vert_1
    \Vert \stateError_0 \Vert_\infty \leq \ 0,\nonumber
\end{align}
}
\hspace{-1mm}where \eqref{Eqn: Apply Inequality for Robust Constraint Satisfaction, k = 0} follows because $\Vert \stateError_0 \Vert_\infty \leq \diam(\paramSet)$.
From \eqref{Eqn: Inequality for Robust Constraint Satisfaction}, we have $\ratePGDBoundSlope_p^\top (\stateNominal_0, \ratePGDNominal_0) + \ratePGDBoundOffset_p \leq 0$.
Thus, $\ratePGD_0, \ratePGDNominal_0 \in [\ratePGDMin, \ratePGDMax]$. 

Suppose for some $k \geq 1$, we have $\Vert (\stateError_j, \ratePGDError_j) \Vert_\infty \leq \tubeSize_j$ for each $j \in [0, k-1]$ and $\ratePGD_{0:k-1}, \ratePGDNominal_{0:k-1} \in [\ratePGDMin, \ratePGDMax]^k$.
For the inductive step, we must prove that $\Vert (\stateError_k, \ratePGDError_k) \Vert_\infty \leq \tubeSize_k$ and $\ratePGD_k, \ratePGDNominal_k \in [\ratePGDMin, \ratePGDMax]$. 
First, we verify that the disturbance bounds prescribed by Prop. \ref{Prop: Quadratic Bound on Linearization Error} are valid for $\disturbance_j = \linearizationError_j(\state_j, \ratePGD_j, \stateNominal_j, \ratePGDNominal_j)$ across all $j \in [0, k-1]$.
Indeed, for each $j \in [0, k-1]$, since $\param_j = \param_0 \in \paramSet$ and $\paramNominal_j = \paramNominal_0 \in \paramSet$:
{
\setlength{\abovedisplayskip}{3.5pt}
\setlength{\belowdisplayskip}{3.5pt}
\begin{subequations}
\begin{align}
    \state_j &= (\var_j, \param_j) \in \varForwardReachableSet{j}(\var_0, \ratePGD_{0:k-1}) \times \paramSet, \\
    \stateNominal_j &= (\varNominal_j, \paramNominal_j) \in \varForwardReachableSet{j}(\varNominal_0, \ratePGDNominal_{0:k-1}) \times \paramSet,
\end{align}
\end{subequations}
}
\hspace{-1mm}where $\ratePGD_{0:k-1}, \ratePGDNominal_{0:k-1} \in [\ratePGDMin, \ratePGDMax]^k$ by the induction hypothesis, and $\var_0, \varNominal_0 \in \varInitialSet$.
Thus, the conditions of Prop. \ref{Prop: Quadratic Bound on Linearization Error} are met at time $j$ for each $j \in [0, k-1]$, so 
\eqref{Eqn: Quadratic Bound on Linearization Error} and \eqref{Eqn: Disturbance Bound, Stacked} imply:
{
\setlength{\abovedisplayskip}{3.5pt}
\setlength{\belowdisplayskip}{3.5pt}
\begin{align}
\label{Eqn: Applying Disturbance Bound, Stacked}
    \disturbance_j 
    \in \Vert (\stateError_j, \ratePGDError_j) \Vert_\infty^2 \hessianBoundStacked \ballUnit_\infty^{\varDim + \paramDim} 
    \subseteq \tubeSize_j \hessianBoundStacked \ballUnit_\infty^{\varDim + \paramDim}
\end{align}
}
\hspace{-1mm}for each $j \in [0, k-1]$. 
Then, by \eqref{Eqn: Applying Disturbance Bound, Stacked}:
{
\setlength{\abovedisplayskip}{3.5pt}
\setlength{\belowdisplayskip}{3.5pt}
\begin{align} \nonumber
    \left\Vert 
    \begin{bmatrix}
        \stateError_k \\
        \ratePGDError_k
    \end{bmatrix}
    \right\Vert_\infty &= \left\Vert 
    \begin{bmatrix}
        (\systemResponse_{\state \disturbance})_{k, 0} \\
        (\systemResponse_{\ratePGD \disturbance})_{k, 0}
    \end{bmatrix}
    \stateError_0
    + \sum_{j=1}^{k}
    \begin{bmatrix}
        (\systemResponse_{\state \disturbance})_{k, j} \\
        (\systemResponse_{\ratePGD \disturbance})_{k, j}
    \end{bmatrix}
    \disturbance_{j-1}
    \right\Vert_\infty \\ \label{Eqn: Apply Inequality for Propagating Recursive Tube Bounds, general k}
    &\leq \left\Vert 
    \begin{bmatrix}
        (\systemResponse_{\state \disturbance})_{k, 0} \\
        (\systemResponse_{\ratePGD \disturbance})_{k, 0}
    \end{bmatrix}
    \right\Vert_\infty
    \Vert \stateError_0 \Vert_\infty
    \\ \nonumber
    &\hspace{1cm} + \sum_{j=1}^{k}
    \tubeSize_{j-1}^2
    \left\Vert
    \begin{bmatrix}
        (\systemResponse_{\state \disturbance})_{k, j} \\
        (\systemResponse_{\ratePGD \disturbance})_{k, j}
    \end{bmatrix}
    \hessianBoundStacked
    \right\Vert_\infty \leq \tubeSize_k,
\end{align}
}
\hspace{-1mm}We also prove robust constraint satisfaction, as follows:
{
\setlength{\abovedisplayskip}{4pt}
\setlength{\belowdisplayskip}{4pt}
\begin{align} \nonumber
    &\ratePGDBoundSlope_p^\top (\state_k, \ratePGD_k) + \ratePGDBoundOffset_p 
    = \ \ratePGDBoundSlope_p^\top (\stateNominal_k, \ratePGDNominal_k) + \ratePGDBoundOffset_p + \ratePGDBoundSlope_p^\top (\stateError_k, \ratePGDError_k) \\ \nonumber
    = \ &\ratePGDBoundSlope_p^\top (\stateNominal_k, \ratePGDNominal_k) + \ratePGDBoundOffset_p \\ \nonumber
    &\hspace{5mm} + 
    \left\Vert 
    \ratePGDBoundSlope_p^\top
    \begin{bmatrix}
        (\systemResponse_{\state \disturbance})_{k, 0} \\
        (\systemResponse_{\ratePGD \disturbance})_{k, 0}
    \end{bmatrix}
    \stateError_0
    + \sum_{j=1}^{k}
    \ratePGDBoundSlope_p^\top
    \begin{bmatrix}
        (\systemResponse_{\state \disturbance})_{k, j} \\
        (\systemResponse_{\ratePGD \disturbance})_{k, j}
    \end{bmatrix}
    \disturbance_{j-1}
    \right\Vert_\infty \\ \nonumber
    = \ &\ratePGDBoundSlope_p^\top (\stateNominal_k, \ratePGDNominal_k) + \ratePGDBoundOffset_p + 
    \left\Vert 
    \ratePGDBoundSlope_p^\top
    \begin{bmatrix}
        (\systemResponse_{\state \disturbance})_{k, 0} \\
        (\systemResponse_{\ratePGD \disturbance})_{k, 0}
    \end{bmatrix}
    \right\Vert_1
    \Vert \stateError_0 \Vert_\infty
    \\
    \label{Eqn: Apply Inequality for Robust Constraint Satisfaction, k geq 1}
    &\hspace{5mm} + \sum_{j=1}^{k} \tubeSize_{j-1}^2
    \left\Vert
    \ratePGDBoundSlope_p^\top
    \begin{bmatrix}
        (\systemResponse_{\state \disturbance})_{k, j} \\
        (\systemResponse_{\ratePGD \disturbance})_{k, j}
    \end{bmatrix}
    \hessianBoundStacked
    \right\Vert_1 
    \leq 0.
\end{align}
}
\hspace{-1mm}Finally, from \eqref{Eqn: Inequality for Robust Constraint Satisfaction}, we again directly obtain that $\ratePGDBoundSlope_p^\top (\stateNominal_k, \ratePGDNominal_k) + \ratePGDBoundOffset_p \leq 0$, thus completing the induction step.
\end{proof}

\looseness=-1
We now present a nonlinear program (NLP) to recover the tightest reachable set for the PGD dynamics \eqref{Eqn: Projected Gradient Descent (PGD)}, as characterized by $\tubeSize_\iterationHorizon$ at the end of the iteration horizon $[0, \iterationHorizon]$, while ensuring that the steplength constraints \eqref{Eqn: Steplength bound constraints} are robustly satisfied\footnote{thus enabling the use of the reachable sets of PGD dynamics for constructing valid outer approximations of $\varOptSet$ via Thm. \ref{Thm: Over-approximating Minimizer Set using Forward Reachable Sets}}. The convex cost $\regularizerSystemResponse(\systemResponse)$ is added as a regularizer, to minimize the uncertainty propagation \cite{Leeman2024FastSLS}: 
{
\setlength{\abovedisplayskip}{3.5pt}
\setlength{\belowdisplayskip}{3.5pt}
\begin{subequations}
\label{Eqn: NLP for Tightest Tube Size}
\begin{align}
\label{Eqn: NLP for Tightest Tube Size, Objective}
    \min_{\stateNominalStacked, \ratePGDNominalStacked, \systemResponse, \tubeSizeStacked} \hspace{5mm} &\tubeSize_\iterationHorizon
    + \regularizerSystemResponse(\systemResponse)
    \\
\label{Eqn: NLP for Tightest Tube Size, Constraints}
    \text{s.t.} \hspace{5mm} &\eqref{Eqn: Nominal Dynamics}, \eqref{Eqn: System Responses, Affine Constraints}, \eqref{Eqn: Inequality for Robust Constraint Satisfaction}, \eqref{Eqn: Inequality for Propagating Recursive Tube Bounds}.
\end{align}
\end{subequations}
}
\hspace{-1.2mm}We tractably compute a feasible point of~\eqref{Eqn: NLP for Tightest Tube Size} using the output feedback SLS solver~\cite{leeman2026vision}, with 
$\mathcal{O}(\iterationHorizon(\varDim^3 + \paramDim^3))$
runtime. 


\subsection{Smoothing PGD Dynamics when Non-Differentiable}
\label{subsec: Smoothing PGD Dynamics when Non-Differentiable}


If $\constraintSet$ is neither $\R^\varDim$ nor an affine subspace of $\R^\varDim$, $\proj_\constraintSet(\cdot)$ is continuous but in general non-differentiable \cite[Ex. 7.4-7.5]{WrightRecht2022OptimizationForDataAnalysis},
in which case
$\dynamicsPGD$ and $\dynamicsFullState$ (see \eqref{Eqn: PGD Updates Notation}) are also, in general, non-differentiable. 
To 
construct
the LTV error system and linearization error bounds in Sec. \ref{subsec: LTV Error System for Smooth PGD Dynamics}--\ref{subsec: Linearization Error Bounds}, we \textit{smooth} $\dynamicsPGD$ and $\dynamicsFullState$, as follows.
Given a \textit{sampling radius} $\samplingRadius > 0$, for 
any
$\state = (\var, \param) \in \R^{\varDim + \paramDim}$ and 
$\ratePGD \in \R$, define $\dynamicsPGDSmoothed(\cdot; \samplingRadius): \R^{\varDim + \paramDim + 1} \ra \R^{\varDim} $ and $\dynamicsFullStateSmoothed(\cdot; \samplingRadius): \R^{\varDim + \paramDim + 1} \ra \R^{\varDim + \paramDim} $ by:
{
\setlength{\abovedisplayskip}{3.5pt}
\setlength{\belowdisplayskip}{3.5pt}
\begin{align}
    \label{Eqn: True Dynamics, Smoothed}
    \dynamicsPGDSmoothed(\state, \ratePGD; \samplingRadius) &:= \E_{\samplingVariable \sim \unif(\ballUnit_2^{\varDim + \paramDim})} \big[ \dynamicsPGD\big( (\state, \ratePGD) + \samplingRadius \samplingVariable \big) \big], \\
    \label{Eqn: Full State Dynamics, Smoothed}
    \dynamicsFullStateSmoothed(\state, \ratePGD; \samplingRadius) &:= \big( \dynamicsPGDSmoothed(\state, \ratePGD; \samplingRadius), \param \big).
\end{align}
}
We impose a fixed $\samplingRadiusMax > 0$ such that $\samplingRadius \in (0, \samplingRadiusMax]$. 

We review properties of smoothed functions, as below \cite{Flaxman2005OnlineConvexOptimization}.

\begin{proposition}[\textbf{Properties of Smoothed Functions}]
\label{Prop: Properties of Smoothed Functions}
Let $\exampleConstraintSet \subseteq \R^\exampleInputDim$, and suppose $\exampleFunction: \R^\exampleInputDim \ra \R^\exampleOutputDim$ is $\lipschitz_\exampleFunction$-Lipschitz continuous in the $\infty$-norm over $\exampleConstraintSet \oplus \samplingRadiusMax I_\exampleInputDim \ballUnit_2^\exampleInputDim$. Define $\exampleFunctionSmoothed(\cdot; \samplingRadius): \R^\exampleInputDim \ra \R^\exampleOutputDim$ as:
{
\setlength{\abovedisplayskip}{3.5pt}
\setlength{\belowdisplayskip}{3.5pt}
\begin{align}
    \exampleFunctionSmoothed(\cdot; \samplingRadius) &:= \E_{\samplingVariable \sim \unif(\ballUnit_2^\exampleInputDim)} \big[ \exampleFunction(\varExample + \samplingRadius \samplingVariable) \big], \hspace{5mm} \forall \ \varExample \in \R^\exampleInputDim.
\end{align}
}
\hspace{-1mm}Then, for each $i \in [\exampleOutputDim]$ and $\varExample, \varExample' \in \exampleConstraintSet$:
\begin{enumerate}
    \item $\Vert \exampleFunctionSmoothed(\varExample; \samplingRadius) - \exampleFunction(\varExample) \Vert_\infty \leq \lipschitz_\exampleFunction \samplingRadius$;
    \item $\nabla \exampleFunctionSmoothed_i(\varExample; \samplingRadius) = \frac{\exampleInputDim}{\samplingRadius} \E_{\samplingVariable \sim \unif(\ballUnit_2^\exampleInputDim)} \big[ \exampleFunction_i (\varExample + \samplingRadius \samplingVariable) \samplingVariable \big]$;
    \item $\Vert \nabla \exampleFunctionSmoothed_i(\varExample; \samplingRadius) - \nabla \exampleFunctionSmoothed_i(\varExample'; \samplingRadius) \Vert_1 \leq \frac{\exampleInputDim^{3/2}}{\samplingRadius} \lipschitz_\exampleFunction \cdot \Vert \varExample - \varExample' \Vert_\infty$.
\end{enumerate}
\end{proposition}

\begin{proof}
The above properties of smoothed functions are direct consequences of \cite[Lemma 1]{Flaxman2005OnlineConvexOptimization} and 
\cite[Sec. 4.2]{Bravo2018BanditLearninginConcaveNPersonGames}.
\end{proof}

\begin{remark}
Prop. \ref{Prop: Properties of Smoothed Functions} indicates that selecting a smaller sampling radius $\samplingRadius$ would produce a smooth proxy $\exampleFunctionSmoothed(\cdot; \samplingRadius)$ that more closely approximates $\exampleFunction$, at the cost of having larger curvature and thus inducing larger linearization errors.
\end{remark}

From Prop. \ref{Prop: Properties of Smoothed Functions}, we obtain the following properties of the smoothed dynamics $\dynamicsFullStateSmoothed$ in \eqref{Eqn: True Dynamics, Smoothed}.

\begin{proposition}
(\textbf{Smoothed PGD Dynamics are Lipschitz and admit Lipschitz Gradients})
\label{Prop: Properties of Smoothed Full Dynamics}
There exists $\lipschitz_\dynamicsFullStateSmoothed, \hessianBoundSmoothed_{i, \samplingRadius}$ such that, $\forall \ (\state, \ratePGD), (\state', \ratePGD') \in \varConvexHull \times \paramSet \times [\ratePGDMin, \ratePGDMax]$ and $i \in [\varDim + \paramDim]$:
{
\setlength{\abovedisplayskip}{3.5pt}
\setlength{\belowdisplayskip}{3.5pt}
\begin{align}
\label{Eqn: Smoothed Dynamics vs. Dynamics, Error Bound}
    \Vert \dynamicsFullStateSmoothed(\state, \ratePGD; \samplingRadius) - \dynamicsFullState(\state, \ratePGD) \Vert_\infty &\leq \lipschitz_\dynamicsFullStateSmoothed \samplingRadius, \\
\label{Eqn: Smoothed Dynamics Gradient, Lipschitz Constant}
    \Vert \nabla \dynamicsFullStateSmoothed_i(\state, \ratePGD; \samplingRadius) - \nabla \dynamicsFullStateSmoothed_i(\state', \ratePGD'; \samplingRadius) \Vert_1 &\leq \hessianBoundSmoothed_{i, \samplingRadius} \Vert (\state, \ratePGD) - (\state', \ratePGD') \Vert_\infty.
\end{align}
}
\end{proposition}

\begin{proof}
\looseness-1Since $\dynamicsFullState$ is continuous, it is $\lipschitz_{\dynamicsFullStateSmoothed}$-Lipschitz continuous over the bounded set $\big((\varConvexHull \times \paramSet) \oplus \samplingRadiusMax I_{\varDim + \paramDim} \ballUnit_2^{\varDim + \paramDim} \big) \times [\ratePGDMin, \ratePGDMax]$. Thus, \eqref{Eqn: Smoothed Dynamics vs. Dynamics, Error Bound}-\eqref{Eqn: Smoothed Dynamics Gradient, Lipschitz Constant} follow directly from Prop. \ref{Prop: Properties of Smoothed Functions}.
\end{proof}


Equipped with
Prop. \ref{Prop: Properties of Smoothed Full Dynamics}, we proceed to adapt the methods in Sec. \ref{subsec: LTV Error System for Smooth PGD Dynamics}-\ref{subsec: Linearization Error Bounds} to settings in which $\dynamicsFullState$ is non-differentiable.
First, we use $\dynamicsFullStateSmoothed$, rather than $\dynamicsFullState$, to re-define the gain and actuation matrices of~\eqref{Eqn: Jacobian Matrices} for the LTV error system as:
{
\setlength{\abovedisplayskip}{3.5pt}
\setlength{\belowdisplayskip}{3.5pt}
\begin{subequations}
\label{Eqn: Jacobian Matrices, with Smoothed Dynamics}
\begin{align}
\label{Eqn: Jacobian Matrix, State, with Smoothed Dynamics}
    \jacobianStateSmoothed_k(\stateNominal_k, \ratePGDNominal_k) &:= \textstyle\frac{\partial \dynamicsFullStateSmoothed}{\partial \state}(\stateNominal_k, \ratePGDNominal_k), \\
    \label{Eqn: Jacobian Matrix, Control, with Smoothed Dynamics}
    \jacobianControlSmoothed_k(\stateNominal_k, \ratePGDNominal_k) &:= \textstyle\frac{\partial \dynamicsFullStateSmoothed}{\partial \ratePGD}(\stateNominal_k, \ratePGDNominal_k).
\end{align}
\end{subequations}
}
\hspace{-1mm}Similarly, we modify the definition of the linearization error $\linearizationErrorSmoothed_k(\state_k, \ratePGD_k, \stateNominal_k, \ratePGDNominal_k) \in \R^{\varDim + \paramDim}$, as given by \eqref{Eqn: Linearization Error}, into:
{
\setlength{\abovedisplayskip}{3.5pt}
\setlength{\belowdisplayskip}{3.5pt}
\begin{align}
\label{Eqn: Linearization Error, with Smoothed Dynamics}
    \linearizationErrorSmoothed_k(\state_k, \ratePGD_k, \stateNominal_k, \ratePGDNominal_k)  & = \dynamicsFullStateSmoothed(\state_k, \ratePGD_k) - \dynamicsFullStateSmoothed(\stateNominal_k, \ratePGDNominal_k) \\ \nonumber 
    &\hspace{5mm} - \jacobianStateSmoothed_k(\stateNominal_k, \ratePGDNominal_k) \stateError_k - \jacobianControlSmoothed_k(\stateNominal_k, \ratePGDNominal_k) \ratePGDError_k,
\end{align}
}
\hspace{-1mm}We also modify the disturbance definition $\disturbance_k$ in \eqref{Eqn: Disturbance}
into $\disturbanceSmoothed_k$,
defined below,
to additionally account for the difference between the true dynamics $\dynamicsFullState$ and its smoothed proxy $\dynamicsFullStateSmoothed$:
{
\setlength{\abovedisplayskip}{3.5pt}
\setlength{\belowdisplayskip}{3.5pt}
\begin{align} 
\label{Eqn: Disturbance, with Smoothed Dynamics}
    \disturbanceSmoothed_k &= \linearizationErrorSmoothed_k(\state_k, \ratePGD_k, \stateNominal_k, \ratePGDNominal_k) + \big[ \dynamicsFullState(\state_k, \ratePGD_k) - \dynamicsFullStateSmoothed(\state_k, \ratePGD_k) \big] \\ \nonumber
    &\hspace{1cm} + \big[ \dynamicsFullStateSmoothed(\stateNominal_k, \ratePGDNominal_k) - \dynamicsFullState(\stateNominal_k, \ratePGDNominal_k) \big].
\end{align}
}
\hspace{-1mm}Using the refined  linearized dynamics and linearization errors \eqref{Eqn: Jacobian Matrices, with Smoothed Dynamics}-\eqref{Eqn: Disturbance, with Smoothed Dynamics},
we derive an expression analogous to \eqref{Eqn: State Error LTV Dynamics} for the state error $\stateError_k := \state_k - \stateNominal_k$, at each iteration $k$:
{
\setlength{\abovedisplayskip}{3.5pt}
\setlength{\belowdisplayskip}{3.5pt}
\begin{align} \nonumber
    \stateError_{k+1} 
    &= \jacobianStateSmoothed_k(\stateNominal_k, \ratePGDNominal_k) \stateError_k + \jacobianControlSmoothed_k(\stateNominal_k, \ratePGDNominal_k) \ratePGDError_k + \disturbanceSmoothed_k.
\end{align}
}

\looseness-1We now adapt Prop. \ref{Prop: Lipschitz Continuity of Gradient of PGD Dynamics}, Prop. \ref{Prop: Quadratic Bound on Linearization Error}, and Thm. \ref{Thm: Robust Steplength Constraint Satisfaction} to the setting in which $\dynamicsFullState$ is non-differentiable.
For Prop. \ref{Prop: Lipschitz Continuity of Gradient of PGD Dynamics}, we replace \eqref{Eqn: Lipschitz Continuity of Gradient of PGD Dynamics} with \eqref{Eqn: Smoothed Dynamics Gradient, Lipschitz Constant}.
For Prop. \ref{Prop: Quadratic Bound on Linearization Error}, we replace $\hessianBound_i$ with $\hessianBoundSmoothed_i$ in \eqref{Eqn: Quadratic Bound on Linearization Error}, i.e.,:
{
\setlength{\abovedisplayskip}{3.5pt}
\setlength{\belowdisplayskip}{3.5pt}
\begin{align}
\label{Eqn: Quadratic Bound on Linearization Error, with Smoothed Dynamics}
    |\linearizationErrorSmoothed_{k,i}(\state_k, \ratePGD_k, \stateNominal_k, \ratePGDNominal_k)| &\leq \hessianBoundSmoothed_i \cdot \Vert (\stateError_k, \ratePGDError_k) \Vert_\infty^2,
\end{align}
}
which can likewise be proved via \eqref{Eqn: Smoothed Dynamics Gradient, Lipschitz Constant} from Prop. \ref{Prop: Lipschitz Continuity of Gradient of PGD Dynamics}.

Next, we modify the disturbance bounds \eqref{Eqn: Disturbance Bound, Stacked} to incorporate the \textit{smoothing error}, i.e., the difference between $\dynamicsFullState$ and $\dynamicsFullStateSmoothed$ as evaluated on the true and nominal trajectories, in addition to the linearization error $\linearizationError_k(\cdot)$:
{
\setlength{\abovedisplayskip}{3.5pt}
\setlength{\belowdisplayskip}{3.5pt}
\begin{align} \nonumber
    \disturbanceSmoothed_k &= \linearizationErrorSmoothed_k(\state_k, \ratePGD_k, \stateNominal_k, \ratePGDNominal_k) + \big[ \dynamicsFullState(\state_k, \ratePGD_k) - \dynamicsFullStateSmoothed(\state_k, \ratePGD_k) \big] \\ \nonumber
    &\hspace{1cm} + \big[ \dynamicsFullStateSmoothed(\stateNominal_k, \ratePGDNominal_k) - \dynamicsFullState(\stateNominal_k, \ratePGDNominal_k) \big] \\ \label{Eqn: Disturbance Bound, Stacked, with Smoothed Dynamics}
    &\in \big( 2 \lipschitz_\dynamicsFullStateSmoothed I_{\varDim + \paramDim} + \Vert (\stateError_k, \ratePGDError_k) \Vert_\infty^2 \hessianBoundSmoothedStacked \big) \ballUnit_\infty^{\varDim + \paramDim},
\end{align}
}
$\hspace{-1mm}$where the value $2 \lipschitz_\dynamicsFullStateSmoothed$ in \eqref{Eqn: Disturbance Bound, Stacked, with Smoothed Dynamics} arises by applying \eqref{Eqn: Smoothed Dynamics vs. Dynamics, Error Bound}, 
and $\hessianBoundSmoothedStacked := \diag\{\hessianBoundSmoothed_1, \cdots, \hessianBoundSmoothed_{\varDim + \paramDim}\}$.

We then modify Thm. \ref{Thm: Robust Steplength Constraint Satisfaction} by respectively replacing \eqref{Eqn: Inequality for Robust Constraint Satisfaction} and \eqref{Eqn: Inequality for Propagating Recursive Tube Bounds} with \eqref{Eqn: Inequality for Robust Constraint Satisfaction, with Smoothed Dynamics} and \eqref{Eqn: Inequality for Propagating Recursive Tube Bounds, with Smoothed Dynamics}, as given below:
{
\setlength{\abovedisplayskip}{3.5pt}
\setlength{\belowdisplayskip}{3.5pt}
\begin{align} 
\label{Eqn: Inequality for Robust Constraint Satisfaction, with Smoothed Dynamics}
    &\ratePGDBoundSlope_p^\top (\stateNominal_k, \ratePGDNominal_k) + \ratePGDBoundOffset_p + \left\Vert \ratePGDBoundSlope_p^\top
    \begin{bmatrix}
        (\systemResponse_{\state \disturbance})_{k, 0} \\
        (\systemResponse_{\ratePGD \disturbance})_{k, 0}
    \end{bmatrix}
    \right\Vert_1 \diam(\paramSet) \\ \nonumber
    &\hspace{5mm} + \sum_{j=1}^k  \left\Vert \ratePGDBoundSlope_p^\top
    \begin{bmatrix}
        (\systemResponse_{\state \disturbance})_{k, 0} \\
        (\systemResponse_{\ratePGD \disturbance})_{k, 0}
    \end{bmatrix}
    \begin{bmatrix}
        2 \lipschitz_\dynamicsFullStateSmoothed I_{\varDim + \paramDim}
        & \tubeSize_{j-1}^2 \hessianBoundSmoothedStacked
    \end{bmatrix}
    \right\Vert_1 \leq 0, \\ \nonumber
    &\hspace{2cm} \forall \ k \in [0, \iterationHorizon], p \in [2], \\ 
\label{Eqn: Inequality for Propagating Recursive Tube Bounds, with Smoothed Dynamics}
    &\hspace{-5pt}\left\Vert 
    \ratePGDBoundSlope_p^\top
    \begin{bmatrix}
        (\systemResponse_{\state \disturbance})_{k, 0} \\
        (\systemResponse_{\ratePGD \disturbance})_{k, 0}
    \end{bmatrix}
    \right\Vert_\infty \diam(\paramSet) \\ \nonumber
    &\hspace{-6pt} +\hspace{-3pt} \sum_{j=1}^k \left\Vert 
    \begin{bmatrix}
        (\systemResponse_{\state \disturbance})_{k, 0} \\
        (\systemResponse_{\ratePGD \disturbance})_{k, 0}
    \end{bmatrix}
    \begin{bmatrix}
        2 \lipschitz_\dynamicsFullStateSmoothed I_{\varDim + \paramDim}
        & \tubeSize_{j-1}^2 \hessianBoundSmoothedStacked
    \end{bmatrix}
    \right\Vert_\infty\hspace{-7pt} \leq \tubeSize_k,\ \ \forall k \in [0, \iterationHorizon].
\end{align}
}
$\hspace{-1mm}$The validity of \eqref{Eqn: Inequality for Robust Constraint Satisfaction, with Smoothed Dynamics} and \eqref{Eqn: Inequality for Propagating Recursive Tube Bounds, with Smoothed Dynamics} as replacements for \eqref{Eqn: Inequality for Robust Constraint Satisfaction} and \eqref{Eqn: Inequality for Propagating Recursive Tube Bounds} in Thm. \ref{Thm: Robust Steplength Constraint Satisfaction}, for non-differentiable dynamics $\dynamicsFullState$, can be verified by retracing the proof of Thm. \ref{Thm: Robust Steplength Constraint Satisfaction}, and applying the modified disturbance bound \eqref{Eqn: Disturbance Bound, Stacked, with Smoothed Dynamics} instead of its original form \eqref{Eqn: Disturbance Bound, Stacked} when appropriate. 
Similarly, we refine the NLP \eqref{Eqn: NLP for Tightest Tube Size}, used to recover the tightest end-of-horizon tube size $\tubeSize_\iterationHorizon$, by replacing \eqref{Eqn: Inequality for Robust Constraint Satisfaction} and \eqref{Eqn: Inequality for Propagating Recursive Tube Bounds} with \eqref{Eqn: Inequality for Robust Constraint Satisfaction, with Smoothed Dynamics} and \eqref{Eqn: Inequality for Propagating Recursive Tube Bounds, with Smoothed Dynamics}, respectively.

%% file: 6_Experiments.tex
\section{Simulations}
\label{sec: Experiments}
Here, we apply Thms. \ref{Thm: Over-approximating Minimizer Set using Forward Reachable Sets}-\ref{Thm: Robust Steplength Constraint Satisfaction} compute over-approximations of the parameterized minimizer sets for specific objectives.





\subsection{Quadratic Optimization}
\label{subsec: ex_quadratic_optim}


Consider the unconstrained minimization of a scalar cost:
{
\setlength{\abovedisplayskip}{3.5pt}
\setlength{\belowdisplayskip}{3.5pt}
\begin{equation}
    \label{eq:scalar_quadratic_cost}
    \min_{\var \in \R} \objSq(\var, \param) := \var^2 + \param \var
\end{equation}
}
\hspace{-1.5mm}over $\var \in \R$,
with parameter uncertainty over $\param \in \paramSetSq := [-0.1, 0.1]$.
$\objSq$ is $m$-strongly convex with $L$-Lipschitz gradients, where
$m = L = 2$.
The PGD dynamics are:
\begin{equation}
    \label{eq:scalar_quadratic_pgd}
    \dynamicsPGDSq(\var, \param, \ratePGD) := \var - \ratePGD (2 \var + \param).
\end{equation}
We 
compute $\nabla^2 \dynamicsFullState_i$ for $i \in \{1, 2\}$ to find that the curvature of $\dynamicsFullState$ can be bounded by $\hessianBound_1=2$, $\hessianBound_2=0$, as in Prop. \ref{Prop: Lipschitz Continuity of Gradient of PGD Dynamics}.

\looseness-1Choosing $\var_0 = 1.0$, $\ratePGDMin=0.4$, $\ratePGDMax=0.6$ and applying Thm. \ref{Thm: Robust Steplength Constraint Satisfaction}, we bound in 
$\infty$-norm
\emph{iterates} of PGD dynamics with horizon $\iterationHorizon=20$. 
By~\eqref{eq:PGD_convergence_rate} and Prop. \ref{Prop: Convergence of PGD}, these iterates converge exponentially to $\varOptSet$ with rate $\convergenceRatePGD = 0.2$.
As our SLS bounds ensure that the iterates are bounded between $\overline{\var} = 10$ and $\underline{\var} = - \overline{\var}$, it can be shown that $\lVert \var_0 - \varOpt(\param) \rVert \le \frac{\lVert \overline{\var} - \underline{\var} \rVert}{1 - \convergenceRatePGD^\iterationHorizon}$ for all $\param \in \paramSet$, overapproximating the distance between the iterates and solution ($\convergenceRadius_k(\var_0)$ in Thm. \ref{Thm: Over-approximating Minimizer Set using Forward Reachable Sets}).
These bounds are combined in Fig.~\ref{fig:unconstrained_quadratic} to overapproximate the parameterized minimizer set of~\eqref{eq:scalar_quadratic_cost}.

\begin{figure}
    \centering
    \includegraphics[width=\linewidth]{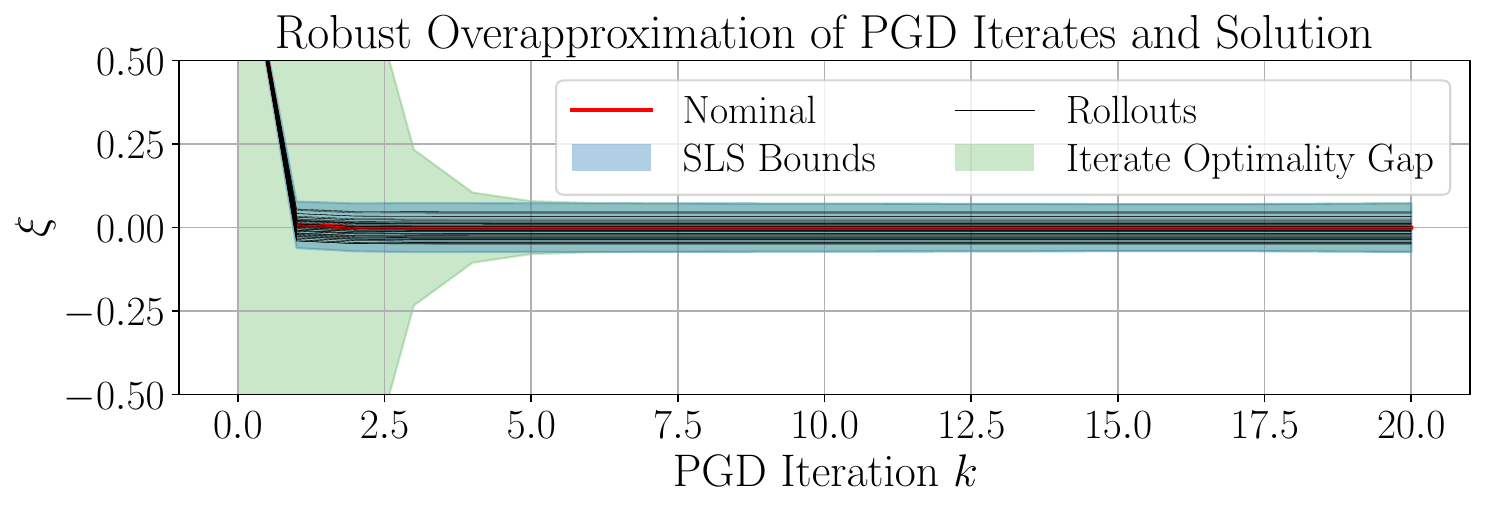}
    \caption{%
        Robust over-approximation of the parameterized solution set of~\eqref{eq:scalar_quadratic_cost}.
        Individual rollouts of PGD corresponding to realizations of $\param$ Monte Carlo sampled from $\paramSetSq$ are shown in black, and the ``nominal'' rollout with $\param = 0$ is shown in red. 
        The blue bounds show the reachable-set bounds computed via Thm. \ref{Thm: Robust Steplength Constraint Satisfaction} that must contain the iterate trajectory for any parameter in $\paramSetSq$. The outer green bounds are guaranteed to contain $\varOptSet$ through Thm. \ref{Thm: Over-approximating Minimizer Set using Forward Reachable Sets}.%
    }
    \label{fig:unconstrained_quadratic}
\end{figure}

\looseness-1
The
minimizer set
for \eqref{eq:scalar_quadratic_cost}
can be explicitly expressed as $\varOptSet = \{ - \frac{\param}{2} \mid \param \in \paramSet \}$.
While our SLS-computed bounds are conservative relative to
$\varOptSet$,
our method can provide tighter bounds than classic robust optimization methods when
no closed-form expression for 
$\varOptSet$ is available
(see Sec. \ref{subsec: ex_lqr}).

\subsection{LQR with Intent Uncertainty}
\label{subsec: ex_lqr}

Consider the following parameterized LQR problem, with variable $\var := (x_0, u_0, \ldots, x_{\LQRHorizon-1}, u_{\LQRHorizon-1}, x_\LQRHorizon) \in \R^{4(T+1) + 2T}$, objective $\objLQR(\var, \param)$, and parameter $\param \in \R$:
{
\setlength{\abovedisplayskip}{3.5pt}
\setlength{\belowdisplayskip}{3.5pt}
\begin{subequations}
\label{eq:lqr}
\begin{align}
\label{eq:lqr_cost}
    \min_{\var} \hspace{3mm} &\frac{1}{2} \textstyle x_\LQRHorizon^\top \LQRStateCost_\LQRHorizon x_\LQRHorizon + \frac{1}{2} \sum_{\LQRIterate=0}^{\LQRHorizon} x_\LQRIterate^\top \LQRStateCost_\LQRIterate x_\LQRIterate + \param u_\LQRIterate^\top \LQRControlCost_\LQRIterate u_\LQRIterate, \\
\label{eq:lqr_dyn}
    &x_{\LQRIterate+1} = \LQRStateDyn_{\LQRIterate} x_{\LQRIterate} + \LQRControlDyn_{\LQRIterate} u_{\LQRIterate}, \hspace{5mm} x_0 = \bar{x}_0,
\end{align}
\end{subequations}
}
\hspace{-1mm}where:
\begin{equation}
    \LQRStateDyn_\LQRIterate = \begin{bsmallmatrix}
        1 & 1 & 0 & 0 \\
        0 & 1 & 0 & 0 \\
        0 & 0 & 1 & 1 \\
        0 & 0 & 0 & 1
    \end{bsmallmatrix}, 
    \LQRControlDyn_\LQRIterate = \begin{bsmallmatrix}
        0 & 0 \\
        1 & 0 \\
        0 & 0 \\
        0 & 1 
    \end{bsmallmatrix}, 
    \LQRStateCost_\LQRIterate = 0.1 I,
    \LQRControlCost_\LQRIterate = 0.1 I.
\end{equation}
In words, $\LQRStateDyn_\LQRIterate$ and $\LQRControlDyn_\LQRIterate$ describe time-invariant double integrator dynamics.
Note that the combined state $\var$ has dimension $4(\LQRHorizon + 1) + 2\LQRHorizon = 64$ for $T=10$.
Above, the parameter uncertainty over $\param \in \paramSetLQR := [0.9, 1.1]$ influences the relative weight of state cost and control cost to 
a path planning agent.
We compactly write the cost~\eqref{eq:lqr_cost} as $\objLQR(\var, \param) = \frac{1}{2} \var^\top H \var = \frac{1}{2} \var^\top (H_\LQRStateCost + \param H_\LQRControlCost) \var$ and dynamics~\eqref{eq:lqr_dyn} as $M \var = b$ for appropriate $H$, $M$, and $b$. 
Concretely,
{
\setlength{\abovedisplayskip}{3.5pt}
\setlength{\belowdisplayskip}{3.5pt}
\begin{align*}
    H_\LQRStateCost &= \blockDiag(\LQRStateCost_0, O, \ldots, \LQRStateCost_{\LQRHorizon-1}, O, \LQRStateCost_{\LQRHorizon}), \\
    H_\LQRControlCost &= \blockDiag(O, \LQRControlCost_0, \ldots, O, \LQRControlCost_{\LQRHorizon-1}, \LQRStateCost_{\LQRHorizon}), 
\end{align*}
}
The constraint manifold is the affine space given by $\{\var = Pz + d: z \in \R^\varDim \}$, with $P := I - M^\top (M M^\top)^{-1}M$ and $d = M^\top (MM^\top)^{-1}b$, 
yielding PGD dynamics 
{
\setlength{\abovedisplayskip}{3.5pt}
\setlength{\belowdisplayskip}{3.5pt}
\begin{equation}
    \label{eq:lqr_pgd}
    \dynamicsPGDLQR(\var, \param, \ratePGD) = P(I - \ratePGD H) \var + d.
\end{equation}
}
\hspace{-3mm}$\objLQR$ is $\min_{\param \in \paramSetLQR} \lambda_{\rm min}(H) = 0.09$-strongly convex and has gradients that are Lipschitz with constant $\max_{\param \in \paramSetLQR} \lambda_{\rm max}(H) = 0.11$.
Since the LQR problem is high-dimensional,
the 
computation of the curvature bounds $\hessianBoundStacked$ is 
much
more challenging than was the case for the
example in Sec. \ref{subsec: ex_quadratic_optim}.
Fixing the region $\overline{\var} = (20, 20, \ldots, 20)$, $\underline{\var} = - \overline{\var}$, $\ratePGDMin=9.9$, $\ratePGDMax=10.1$ (guaranteed by SLS to contain the true iterates), we 
compute a guaranteed local over-approximation using the interval arithmetic toolbox {\tt immrax}~\cite{immrax}.
Then,
as detailed
in Sec. \ref{subsec: ex_quadratic_optim}, 
we compute bounds on the iterates $\xi_k$ of PGD minimizing~\eqref{eq:lqr_cost}
across $\iterationHorizon = 10$ iterations, and bloat these bounds using the convergence rate $\convergenceRatePGD = 0.111$.
In Fig. \ref{fig:lqr_slices}, we visualize components of the resulting anytime over-approximations of
$\varOptSet$
corresponding to position states at LQR timesteps 0, 3, and 6.
Our bounds are only slightly larger than 
those indicated by
Monte Carlo sampled rollouts. 

\begin{figure}
    \centering
    \includegraphics[width=\linewidth]{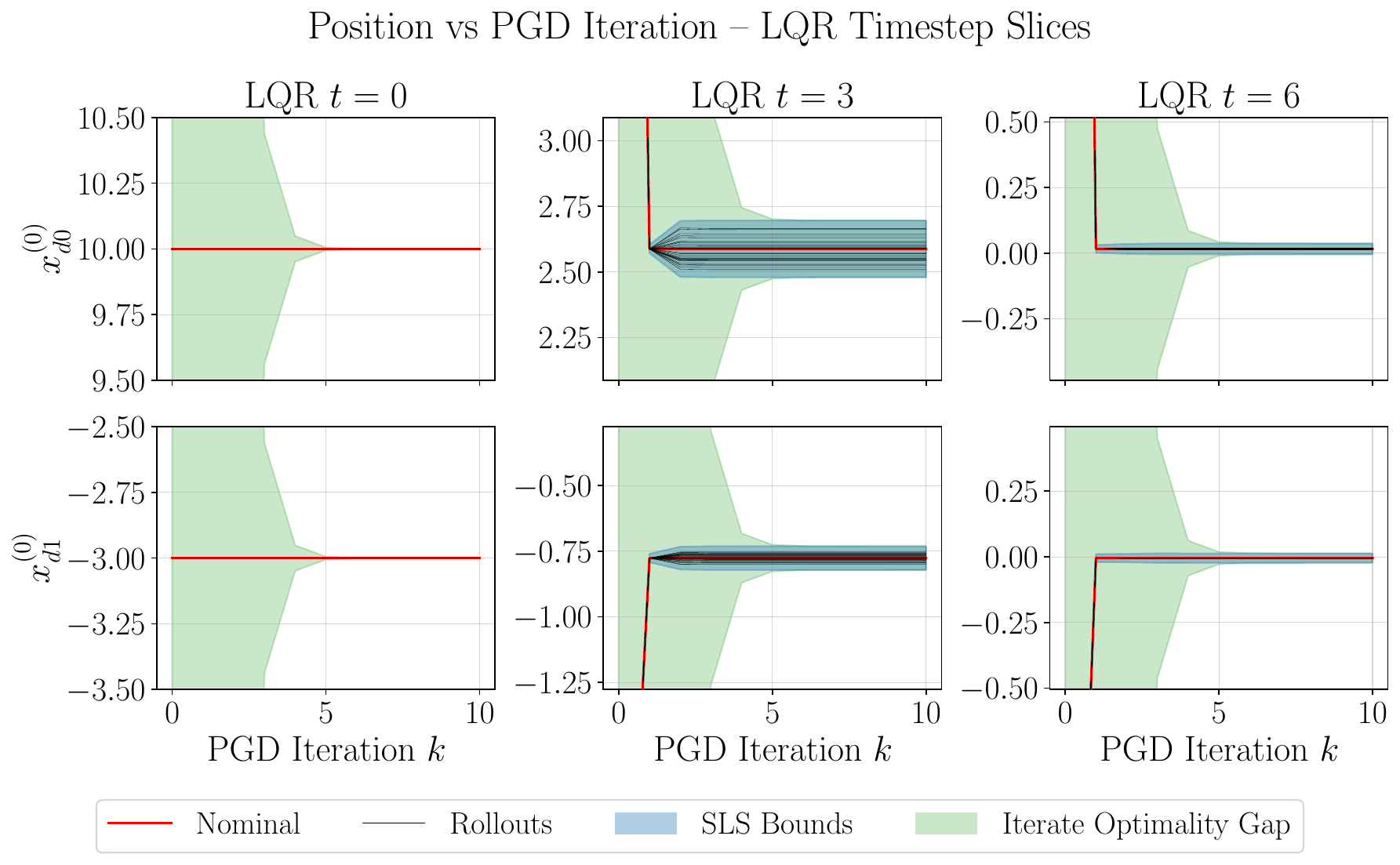}
    \caption{%
        Robust over-approximation of the planned positions of a parameterized LQR system.
        All visual elements are as in Fig. \ref{fig:unconstrained_quadratic}.
        For space, we only show the components of $\var$ corresponding to position on timesteps $\LQRIterate = 0, 3, 6$.%
        }
    \label{fig:lqr_slices}
\end{figure}

\subsubsection{Baseline Comparison}
\label{subsubsec:baseline}
Compared to our method,
classical Implicit Function Theorem methods~\cite[Section 11.5.2]{WrightRecht2022OptimizationForDataAnalysis}
perform poorly.
For
dual variables
$\varDual$,
differentiating the KKT conditions of 
\eqref{eq:lqr} at an optimal $(\varOpt, \varDualOpt)$, yields:
{
\setlength{\abovedisplayskip}{3.5pt}
\setlength{\belowdisplayskip}{3.5pt}
\begin{align*}
    \begin{bmatrix}
        H & M^\top \\
        M & O
    \end{bmatrix}
    \begin{bmatrix}
        \nabla_\param \varOpt(\param) \\ 
        \nabla_\param \varDualOpt(\param)
    \end{bmatrix} = 
    \begin{bmatrix}
        - H_\LQRControlCost \var \\ 
        \zero
    \end{bmatrix},
\end{align*}
}
\hspace{-1mm}with solution $\nabla_\param \varOpt = [ H^{-1} M^\top (M H^{-1} M^\top)^{-1} M H^{-1} - H^{-1} ] H_\LQRControlCost \var := \mathcal{P} H_\LQRControlCost \var$.
Over the parameter region $[\underline{\var}, \overline{\var}]$ guaranteed to be robustly forward invariant for PGD by SLS, $L_{\varOpt} := \sup_{\var \in [\underline{\var}, \overline{\var}], \param \in \paramSetLQR} \lVert \nabla_\param \varOpt \rVert_2$ is a local Lipschitz constant for $\varOpt(\param)$. 
By sub-multiplicativity of norms and standard linear algebraic techniques, for any $\param_1, \param_2 \in \paramSetLQR$
{
\setlength{\abovedisplayskip}{3.5pt}
\setlength{\belowdisplayskip}{3.5pt}
\begin{align*}
    \lVert \varOpt(\param_1) - \varOpt(\param_2) &\rVert_2 \le L_\varOpt \lVert \param_1 - \param_2 \rVert_2 \\
    & \le \lVert \mathcal{P} \rVert_2 \lVert H_\LQRControlCost \rVert_2 \diam(\paramSetLQR) \textstyle\sup_{\var \in [\underline{\var}, \overline{\var}]} \lVert \var \rVert_2 \\
    & = \lVert H^{-1} \rVert_2 \lVert H_\LQRControlCost \rVert_2 \diam(\paramSetLQR) \lVert \overline{\var} \rVert_2  \approx 35.
\end{align*}
}
\hspace{-1mm}The maximum bound width computed by our method for \emph{any} component of $\var$ is $0.2$, which is much less conservative. 
We note that
an analytical expression for $\varOptSet$ is unavailable 
since the LQR dynamics impose multiple equality constraints.

\subsection{Constrained Quadratic Optimization}
\label{subsec:ex_constrained_quadratic}

Consider the following constrained variant of~\eqref{eq:scalar_quadratic_cost}:
{
\setlength{\abovedisplayskip}{3pt}
\setlength{\belowdisplayskip}{3pt}
\begin{subequations}
\label{eq:scalar_quadratic_constrained_cost}
\begin{align}
    \min_{\var \in \R} \hspace{5mm} & \objSqCons(\var, \param) := \SqcQuadScale \var^2 + \SqcLinScale \param \var, \\
    \text{s.t.} \hspace{5mm} &\var \ge \SqcConsVal.
\end{align}
\end{subequations}
}

Since the constraint 
set above
is \emph{not} an affine subspace, the PGD dynamics are non-smooth, so we apply the smoothing 
process in Sec.~\ref{subsec: Smoothing PGD Dynamics when Non-Differentiable} to obtain 
reachable set bounds.

The objective satisfies Assumption~\ref{Assump: Objective Properties} with $m = L = 2 \SqcQuadScale$.
The gradient descent dynamics without projection are similar to~\eqref{eq:scalar_quadratic_pgd}, and have Lipschitz constant $\SqcLipschitz = \lvert 1 - 2 \SqcQuadScale \ratePGD \rvert + \lvert \SqcLinScale \ratePGD \rvert + \lvert 2 \SqcQuadScale \var + \SqcLinScale \param \rvert$ 
as in Prop~\ref{Prop: Properties of Smoothed Functions}. 
Since this problem is scalar, projection is non-expansive in every norm, and $\SqcLipschitz$ is also a Lipschitz constant for PGD dynamics.
Applying Prop~\ref{Prop: Properties of Smoothed Functions}, points 3) and 1) with sampling radius $\samplingRadius = 0.1$ bound the curvature of the smoothed dynamics and their distance from the true PGD update. 
Then, \eqref{Eqn: Disturbance Bound, Stacked, with Smoothed Dynamics}--\eqref{Eqn: Inequality for Propagating Recursive Tube Bounds, with Smoothed Dynamics} 
over-approximate
PGD iterates for any parameterized realization of~\eqref{eq:scalar_quadratic_constrained_cost}. 

Using $\var_0 = 0.5$, 
$\paramSet = [-0.1, 0.1]$, 
$\SqcQuadScale = 0.0521$, $\SqcLinScale = 0.0054$, $\ratePGDMin = 9.54$, and $\ratePGDMax = 9.56$,
we outer-approximate
$\varOptSet$ for~\eqref{eq:scalar_quadratic_constrained_cost} 
(see Fig. \ref{fig:constrained_quadratic}). 
Due to error between the true nonsmooth PGD update and the smoothed dynamics used for SLS, each PGD iteration incurs additional conservativeness. 
Rollouts
quickly approach the constraint boundary, causing the over-approximation to contain non-differentiable points for 
most
of the iteration horizon;
we thus require
Prop.~\ref{Prop: Properties of Smoothed Functions} to bound the curvature of the smoothed dynamics. 
Both factors increase conservativeness compared to previous examples. 

\begin{figure}
    \centering
    \includegraphics[width=\linewidth]{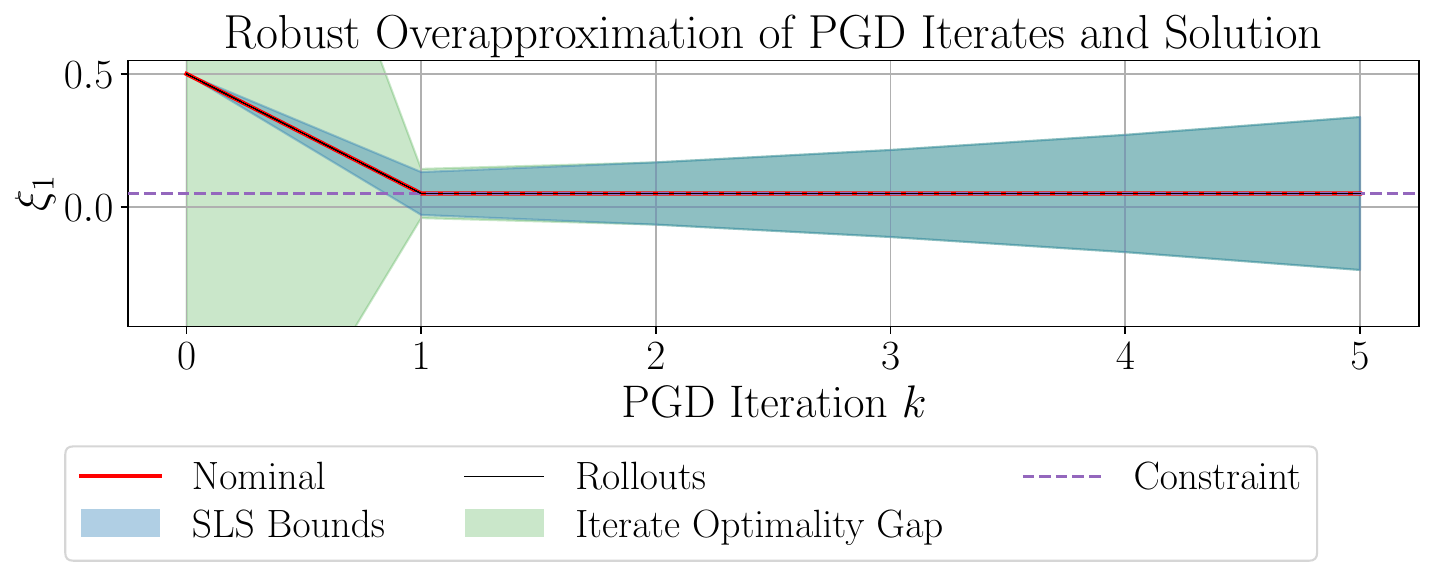}
    \caption{%
        Robust over-approximation of the parameterized minimizer set of~\eqref{eq:scalar_quadratic_constrained_cost}.
        All visual elements are as in Fig.~\ref{fig:unconstrained_quadratic}. 
        The boundary of the constraint set $\constraintSet$ is also shown in purple, for convenience. 
    }
    \label{fig:constrained_quadratic}
\end{figure}





%% file: 7_Conclusion_Future_Work.tex
\section{Conclusion}
\label{sec: Conclusion and Future Work}

We present a novel system-level synthesis (SLS)-based method for over-approximating the minimizer set of strongly convex objective costs characterized by uncertain parameters.
Our method leverages the convergence properties of projected gradient descent (PGD) dynamics over convex optimization landscapes, to recast the problem of over-approximating the minimizer set as that of computing the PGD dynamics' forward reachable sets.
By considering the cost parameter as constant but unknown, we model PGD dynamics as a nonlinear system with initial state uncertainty, and refine existing nonlinear SLS techniques to compute its forward reachable sets.
Our numerical results illustrate that our method outperforms sensitivity analysis and fixed-steplength PGD baselines in over-approximating the minimizer set with low conservativeness.
Future work will fuse our method for minimizer set over-approximation with adaptive control techniques for shrinking $\Theta$ online from observations. 

